\documentclass[12pt,onecolumn, draftclsnofoot]{IEEEtran}
\usepackage{bbm}
\usepackage{mathrsfs}
\usepackage{amsfonts}
\usepackage{epsfig}
\usepackage{psfrag}
\usepackage{graphicx}
\usepackage{epstopdf}
\usepackage{cite}
\usepackage{caption}

\usepackage{array}
\usepackage{stfloats}
\usepackage{textcomp}
\usepackage{verbatim}
\usepackage[caption=false,font=normalsize,labelfont=sf,textfont=sf]{subfig}
\usepackage[justification=centering]{caption}
\usepackage{multirow}
\usepackage{makecell}

\usepackage{amssymb}
\usepackage{amsthm}
\usepackage{algorithm,algorithmic}
\usepackage[tbtags]{amsmath}
\usepackage{url}
\usepackage{bm}

\usepackage{color}
\definecolor{gray}{RGB}{128,128,128}

\allowdisplaybreaks

\newtheorem{theorem}{Theorem}
\newtheorem{assumption}{Assumption}

\newtheorem{lemma}{Lemma}
\newtheorem{corollary}{Corollary}

\newtheorem{remark}{Remark}

\renewcommand{\IEEEproof}{\textbf{Proof}.}

\usepackage{color}

\definecolor{gray}{RGB}{128,128,128}

\allowdisplaybreaks

\begin{document}
\title{\textbf{Distributed Event-Triggered Bandit Convex Optimization with Time-Varying Constraints}}
\author{Kunpeng~Zhang,
        Xinlei~Yi,
        Guanghui~Wen,\\
        Ming~Cao,
        Karl~H.~Johansson,
        Tianyou~Chai,
        and Tao~Yang
\thanks{K. Zhang, T. Chai and T. Yang are with the State Key Laboratory of Synthetical Automation for Process Industries, Northeastern University, Shenyang 110819, China {\tt\small 2110343@stu.neu.edu.cn; \{tychai; yangtao\}@mail.neu.edu.cn}}%
\thanks{X. Yi is with the Laboratory for Information and Decision Systems, Massachusetts Institute of Technology, Cambridge, MA 02139, USA {\tt\small xinleiyi@mit.edu}}%
\thanks{G. Wen is with the Department of Mathematics, Southeast University, Nanjing 210096, China {\tt\small  ghwen@seu.edu.cn}}%
\thanks{M. Cao is with the Engineering and Technology Institute Groningen, Faculty of Science and Engineering, University of Groningen, AG 9747 Groningen, The Netherlands {\tt\small m.cao@rug.nl}}%
\thanks{K. H. Johansson is with Division of Decision and Control Systems, School of Electrical Engineering and Computer
Science, KTH Royal Institute of Technology, and he is also affiliated with Digital Futures, 10044, Stockholm, Sweden {\tt\small kallej@kth.se}}%
}


\maketitle

\begin{abstract}
This paper considers the distributed bandit convex optimization problem with time-varying inequality constraints over a network of agents, where the goal is to minimize network regret and cumulative constraint violation.
Existing distributed online algorithms require that each agent broadcasts its decision to its neighbors at each iteration.
To better utilize the limited communication resources, we propose a distributed event-triggered online primal--dual algorithm with two-point bandit feedback.
Under several classes of appropriately chosen decreasing parameter sequences and non-increasing event-triggered threshold sequences, we establish dynamic network regret and network cumulative constraint violation bounds.
These bounds are comparable to the results achieved by distributed event-triggered online algorithms with full-information feedback.
Finally, a numerical example is provided to verify the theoretical results.
\end{abstract}



\section{INTRODUCTION}
Bandit convex optimization has drawn a growing attention due to its broad applications such as online routing in data networks and online advertisement placement in web search~\cite{Hazan2016a}.
Different from online convex optimization, where the decision maker receives full-information feedback for the loss function at each iteration \cite{Zinkevich2003, Hazan2007, Mokhtari2016, Yuan2018, Zhang2018, Guo2022, Zhang2024}, i.e., the loss function is revealed to the decision maker at each iteration,
in bandit convex optimization the decision maker receives bandit feedback for the loss function at each iteration, i.e., only the values of the loss function at some points are revealed to the decision maker at each iteration.
In general, regret is a common performance metric \cite{ShalevShwartz2012}, which measures the difference of the cumulative loss between the decision sequence selected by the decision maker and a comparator sequence.
When each element of the comparator sequence is the offline optimal static decision, this metric is called \textrm{static regret} \cite{Yang2016a, Yi2021b}.
When the comparator sequence is the offline optimal dynamic decision sequence, this metric is called \textrm{dynamic regret}.

Bandit convex optimization with time-invariant constraints is well studied.
For example, the authors of \cite{Flaxman2005} propose a projection-based online gradient descent algorithm with one-point bandit feedback, and establish an ${\cal O}({T^{3/4}})$ static regret bound for convex loss functions.
The authors of \cite{Agarwal2010} propose a projection-based online gradient descent algorithm by introducing the notable two-point bandit feedback, and establish an ${\cal O}(\sqrt T )$ static regret bound for convex loss functions.
The authors of~\cite{Mahdavi2012} consider the scenarios where constraints are characterized by static inequalities,
and introduce the idea of long-term constraints (i.e., inequality constraints are permitted to be violated but are fulfilled in the long run) to avoid projection operations onto the inequality constrained set due to the high computational complexity.
In contrast, bandit convex optimization with time-varying constraints are studied in \cite{Chen2019, Cao2019}, where constraints are characterized by time-varying inequalities. Different from the time-invariant constraint setting, where the decision maker knows the constrained set in advance when she makes a decision, in the time-varying constraint setting considered in \cite{Chen2019, Cao2019} the decision maker has no a priori knowledge of the current inequality constrained set, and the information of inequality constraint functions is revealed along with the values of the loss function after she makes a decision.

The aforementioned studies concentrate on centralized online algorithms with bandit feedback, which suffer a plethora of limitations, e.g., single point of failure, heavy communication and computation overhead \cite{Yang2019, Li2023}.
To deal with these limitations, distributed online algorithms with bandit feedback are developed in \cite{Yuan2021, Yuan2022a, Yuan2022, Cao2021a, Cao2022, Yi2023, Yan2024}.
Along the line of the time-invariant constraint setting, the authors of \cite{Yuan2021, Yuan2022a} propose the projection-based distributed online algorithms with two-point and one-point bandit feedback.
In the presence of feedback delays, the authors of \cite{Cao2021a} propose a projection-based distributed online algorithm with two-point bandit feedback, and analyze the impact of delay size on the algorithm performance.
In addition, the authors of \cite{Yuan2022, Cao2022, Yan2024} consider static inequality constraints, and use the idea of long-term constraints to reduce computational burden of projection operations.
For the time-varying constraint setting, the authors of \cite{Yi2023} propose a distributed online primal--dual algorithm with two-point bandit feedback by using two-point stochastic subgradient approximations for both loss and constraint functions at each iteration.

Note that in the above studies on distributed online algorithms with bandit feedback, all the decision makers require to collaboratively make decisions through local information exchange with their neighbors at each iteration.
However, communication resources are limited that frequent communication among the decision makers may cause network congestion.
To better utilize communication resources, the authors of \cite{Cao2021} propose a distributed event-triggered online algorithm with two-point bandit feedback, where each agent broadcasts the current local decision to its neighbors only if the norm of the difference between the current decision and its last broadcasted decision is not less than the current event-triggering threshold. Moreover, sublinear static regret is achieved when the event-triggering threshold sequence is non-increasing and converges to zero.
By using one-point and two-point stochastic subgradient estimators, two distributed event-triggered online algorithms with delayed bandit feedback are developed in \cite{Xiong2023}, and static regret bounds are established.

The existing studies on distributed event-triggered online algorithms with bandit feedback do not consider inequality constraints.
In this context,
this paper studies the distributed bandit convex optimization problem with time-varying constraints, where the decision makers receive bandit feedback for both loss and constraint functions at each iteration.
The contributions are summarized as follows.\vspace{-1pt}
\begin{itemize}
\item[$\bullet$]
This paper proposes a distributed event-triggered online primal--dual algorithm with two-point bandit feedback by integrating event-triggered communication with the distributed online algorithm in \cite{Yi2023}.
Note that the introduction of event-triggered communication causes nontrivial challenges for performance analysis, which will be explained in detail in Remark~2.
The proposed algorithm can be viewed as a bandit version of the distributed event-triggered online algorithm with full-information feedback in \cite{Zhang2024}.
Their proofs are significantly different, which will be also explained in detail in Remark~2.
Note that the authors in \cite{Cao2021, Xiong2023} do not consider inequality constraints and analyze static regret, and we consider time-varying inequality constraints and analyze dynamic regret. Moreover, the proposed algorithm only uses bandit feedback for constraint functions.

\item[$\bullet$]
When the step-size sequence of local primal variables is appropriately designed based on the event-triggering threshold sequence, this paper establishes dynamic network regret and network cumulative constraint violation bounds for the proposed algorithm under a non-increasing event-triggered threshold sequence (see Theorem~1).
The bounds would be sublinear if the path-length of the comparator sequence (i.e., the accumulated dynamic variation of the comparator sequence) grows sublinearly and the event-triggering threshold sequence converges to zero.
In addition, this paper also establishes dynamic network regret and network cumulative constraint violation bounds under the event-triggering threshold sequence produced by $\{ 1/{t^\theta } \}$ with $\theta  > 0$ (see Corollary~1) and the event-triggering threshold sequence produced by $\{1/{c^t}\}$ with $c  > 1$ (see Corollary~2).
These bounds are the same as the results achieved by the distributed event-triggered online algorithm with full-information feedback in \cite{Zhang2024}.
When event-triggered communication is not considered, these bounds recover the results achieved by the centralized online algorithm with two-point bandit feedback in \cite{Cao2019}.
When inequality constraints are not considered, these dynamic network regret bounds recover the results achieved by the distributed event-triggered online algorithms with two-point bandit feedback in \cite{Cao2021, Xiong2023}.
When both event-triggered communication and inequality constraints are not considered, these dynamic network regret bounds recover the results achieved by the centralized online algorithm with two-point bandit feedback in \cite{Agarwal2010} and the distributed online algorithm with two-point bandit feedback in \cite{Yuan2021}.
\item[$\bullet$]
When the step-size sequence of local primal variables is independently designed, this paper establishes dynamic network regret and network cumulative constraint violation bounds (see Theorem~2).
These bounds are the same as the results achieved by the distributed event-triggered online algorithm with full-information feedback in \cite{Zhang2024}.
When event-triggered communication is not considered, these bounds recover the results achieved by the distributed online algorithm with two-point bandit feedback in \cite{Yi2023}.
Note that this paper is among the first to establish dynamic network regret bounds (and network cumulative constraint violation bounds) for distributed event-triggered bandit convex optimization (with time-varying constraints).
\end{itemize}

\begin{table*}
\centering
\caption{Comparison of this paper to related works on bandit convex optimization.}
\scalebox{0.9}{
\begin{tabular}{c|c|c|c|c|c|c}
\Xcline{1-7}{1pt}
\multicolumn{2}{c|}{\multirow{2}{*}{Reference}} & {\multirow{2}{*}{Problem type}} & \multirow{2}{*}{Constraint type} & \multirow{2}{*}{Information feedback} & \multirow{2}{*}{Event-triggering} & \multirow{2}{*}{Regret type}\\
\multicolumn{2}{c|}{} & & & & & \\
\cline{1-7}
\multicolumn{2}{c|}{\multirow{2}{*}{\cite{Agarwal2010}}} & {\multirow{2}{*}{Centralized}} & \multirow{2}{*}{${g_t}( x ) \equiv {\mathbf{0}_m}$} & \multirow{2}{*}{\makecell{Two-point bandit feedback for ${f_t}$}} & \multirow{2}{*}{No} & \multirow{2}{*}{Static regret}\\
\multicolumn{2}{c|}{} & & & &  &\\
\cline{1-7}
\multicolumn{2}{c|}{\multirow{2}{*}{\cite{Chen2019}}} & {\multirow{2}{*}{Centralized}} & \multirow{2}{*}{\makecell{${g_t}( x ) \le {\mathbf{0}_m}$  \\ and Slater's condition}} & \multirow{2}{*}{\makecell{Two-point bandit feedback for ${f_t}$, and $\nabla {g_t}$}} & \multirow{2}{*}{No} & \multirow{2}{*}{Dynamic regret}\\
\multicolumn{2}{c|}{} & & & & & \\
\cline{1-7}
\multicolumn{2}{c|}{\multirow{2}{*}{\cite{Cao2019}}} & {\multirow{2}{*}{Centralized}} & \multirow{2}{*}{${g_t}( x ) \le {\mathbf{0}_m}$ } & \multirow{2}{*}{\makecell{Two-point bandit feedback for ${f_t}$ and $ {g_t}$}} & \multirow{2}{*}{No} & \multirow{2}{*}{Dynamic regret}\\
\multicolumn{2}{c|}{} & & & & & \\
\cline{1-7}
\multicolumn{2}{c|}{\multirow{2}{*}{\cite{Yuan2021}}} & {\multirow{2}{*}{Distributed}} & \multirow{2}{*}{${g_t}( x ) \equiv {\mathbf{0}_m}$} & \multirow{2}{*}{\makecell{Two-point bandit feedback for ${f_t}$}} & \multirow{2}{*}{No} & \multirow{2}{*}{Static regret}\\
\multicolumn{2}{c|}{} & & & &  &\\
\cline{1-7}
\multicolumn{2}{c|}{\multirow{2}{*}{\cite{Yi2023}}} & {\multirow{2}{*}{Distributed}} & \multirow{2}{*}{${g_t}( x ) \le {\mathbf{0}_m}$} & \multirow{2}{*}{\makecell{Two-point bandit feedback for ${f_t}$ and ${g_t}$}} & \multirow{2}{*}{No} & \multirow{2}{*}{Dynamic regret}\\
\multicolumn{2}{c|}{} & & & & & \\
\cline{1-7}
\multicolumn{2}{c|}{\multirow{2}{*}{\cite{Cao2021}}} & {\multirow{2}{*}{Distributed}} & \multirow{2}{*}{${g_t}( x ) \equiv {\mathbf{0}_m}$} & \multirow{2}{*}{\makecell{Two-point bandit feedback for ${f_t}$}} & \multirow{2}{*}{Yes} & \multirow{2}{*}{Static regret}\\
\multicolumn{2}{c|}{} & & & & & \\
\cline{1-7}
\multicolumn{2}{c|}{\multirow{2}{*}{\cite{Xiong2023}}} & {\multirow{2}{*}{Distributed}} & \multirow{2}{*}{${g_t}( x ) \equiv {\mathbf{0}_m}$} & \multirow{2}{*}{\makecell{Two-point bandit feedback for ${f_t}$}} & \multirow{2}{*}{Yes} & \multirow{2}{*}{Static regret}\\
\multicolumn{2}{c|}{} & & & & & \\
\cline{1-7}
\multicolumn{2}{c|}{\multirow{2}{*}{This paper}} & {\multirow{2}{*}{Distributed}} & \multirow{2}{*}{${g_t}( x ) \le {\mathbf{0}_m}$} & \multirow{2}{*}{\makecell{Two-point bandit feedback for ${f_t}$ and ${g_t}$}} & \multirow{2}{*}{Yes} & \multirow{2}{*}{Dynamic regret}\\
\multicolumn{2}{c|}{} & & & & & \\
\Xcline{1-7}{1pt}
\end{tabular}}
\end{table*}

The detailed comparison of this paper to related studies is summarized in TABLE~I.

The remainder of this paper is organised as follows.
Section~II presents the problem formulation and motivation.
Section~III proposes the distributed event-triggered online primal--dual algorithm with two-point bandit feedback, and analyzes its performance.
Section~IV demonstrates a numerical simulation to verify the theoretical results.
Finally, Section~V concludes this paper.
All proofs are given in Appendix.

\textbf{Notations:} ${\mathbb{N}_ + }$, $\mathbb{R}$, ${\mathbb{R}^p}$ and $\mathbb{R}_ + ^p$ denote the sets of all positive integers, real numbers, $p$-dimensional and nonnegative vectors, respectively. Given $m \in {\mathbb{N}_ + }$, $[ m ]$ denotes the set $\{ {1, \cdot  \cdot  \cdot ,m} \}$. Given vectors $x$ and $y$, ${x^T}$ denotes the transpose of the vector $x$, and $\langle {x,y} \rangle $ and $x \otimes y$ denote the standard inner and Kronecker product of the vectors $x$ and $y$, respectively. ${\mathbf{0}_m}$ denotes the $m$-dimensional column vector whose components are all $0$. $\mathrm{col}( {q_1}, \cdot  \cdot  \cdot ,{q_n} )$ denotes the concatenated column vector of ${q_i} \in {\mathbb{R}^{{m_i}}}$ for $i \in [ n ]$. ${\mathbb{B}^p}$ and ${\mathbb{S}^p}$ denote the unit ball and sphere centered around the origin in ${\mathbb{R}^p}$ under Euclidean norm, respectively. $\mathbf{E}$ denotes the expectation. For a set $\mathbb{K} \in {\mathbb{R}^p}$ and a vector $ x \in {\mathbb{R}^p}$, ${\mathcal{P}_{\mathbb{K}}}(  x  )$ denotes the projection of the vector $x$ onto the set $\mathbb{K}$, i.e., ${\mathcal{P}_{\mathbb{K}}}( x ) = \arg {\min _{y \in {\mathbb{K}}}}{\| {x - y} \|^2}$, and $[  x  ]_+$ denotes ${\mathcal{P}_{\mathbb{R}_ + ^p}}( x )$. For a function $f$ and a vector $ x $, $\partial f( x )$ denotes the subgradient of $f$ at $x$.

\section{PROBLEM FORMULATION AND MOTIVATION}
Consider the distributed bandit convex optimization problem with time-varying constraints.
At iteration $t$, a network of $n$ agents is modeled by a time-varying directed graph ${\mathcal{G}_t} = ( {\mathcal{V},{\mathcal{E}_t}} )$ with the agent set $\mathcal{V} = [ n ]$ and the edge set ${\mathcal{E}_t} \subseteq \mathcal{V} \times \mathcal{V}$. $( {j,i} ) \in {\mathcal{E}_t}$ indicates that agent $i$ can receive information from agent $j$.
The sets of in- and out-neighbors of agent $i$ are $\mathcal{N}_i^{\text{in}}( {{\mathcal{G}_t}} ) = \{ {j \in [ n ]|( {j,i} ) \in {\mathcal{E}_t}} \}$ and $\mathcal{N}_i^{\text{out}}( {{\mathcal{G}_t}} ) = \{ {j \in [ n ]|( {i,j} ) \in {\mathcal{E}_t}} \}$, respectively.
An adversary first erratically selects $n$ local convex loss functions $\{ {{f_{i,t}}:\mathbb{X} \to \mathbb{R}} \}$ and $n$ local convex constraint functions $\{ {{g_{i,t}}:\mathbb{X} \to {\mathbb{R}^{{m_i}}}} \}$ for $i \in [n]$, where $\mathbb{X} \subseteq {\mathbb{R}^p}$ is a known set, and both ${m_i}$ and $p$ are positive integers. Then, the agents collaborate to select their local decisions $\{ {{x_{i,t}} \in \mathbb{X}} \}$ without prior access to $\{ {{f_{i,t}}} \}$ and $\{ {{g_{i,t}}} \}$. At the same time, the values of ${f_{i,t}}$ and ${g_{i,t}}$ at the point ${x_{i,t}}$ as well as at other potential points are privately revealed to agent $i$.
The goal of the network is to choose the decision sequence $\{ {{x_{i,t}}} \}$ for $i \in [n]$ and $t \in [T]$ such that both network regret
\begin{flalign}
{\rm{Net}\mbox{-}\rm{Reg}}( {\{ {{x_{i,t}}} \},{y_{[ T ]}}} ) &:= \frac{1}{n}\sum\limits_{i = 1}^n {\sum\limits_{t = 1}^T {{f_t}( {{x_{i,t}}} )} }  - \sum\limits_{t = 1}^T {{f_t}( {{y_t}} )}, \label{regret-eq1}
\end{flalign}
and network cumulative constraint violation
\begin{flalign}
{\rm{Net}\mbox{-}\rm{CCV}}( {\{ {{x_{i,t}}} \}} ) &:= \frac{1}{n}\sum\limits_{i = 1}^n {\sum\limits_{t = 1}^T {\| {{{[ {{g_t}( {{x_{i,t}}} )} ]}_ + }} \|} }, \label{CCV-eq2}
\end{flalign}
increase sublinearly, where ${y_{\left[ T \right]}} = \left( {{y_1}, \cdot  \cdot  \cdot ,{y_T}} \right)$ is a comparator
sequence, ${f_t}( x ) = \frac{1}{n}\sum\nolimits_{j = 1}^n {{f_{j,t}}( x )} $ and ${g_t}( x ) = {\rm{col}}\big( {{g_{1,t}}( x ), \cdot  \cdot  \cdot ,{g_{n,t}}( x )} \big) \in {\mathbb{R}^m}$ are the global loss and constraint functions of the network at iteration $t$, respectively, and $m = \sum\nolimits_{i = 1}^n {{m_i}} $.

Note that network regret \eqref{regret-eq1} measures the difference of the network-wide cumulative loss between the decision sequence $\{ {{x_{i,t}}} \}$ and the comparator sequence ${y_{[ T ]}}$, which is also used in \cite{Zhang2024, Yi2023}. This regret is different with the used regrets in \cite{Yuan2021, Cao2021, Xiong2023} that measure the difference of the cumulative loss between the decision sequence $\{ {{x_{i,t}}} \}$ of a single agent and the comparator sequence ${y_{[ T ]}}$. Because the established bounds of the used regrets in \cite{Yuan2021, Cao2021, Xiong2023} are uniform bounds for all the agents, we can compare the established bounds of network regret in this paper with those in \cite{Yuan2021, Cao2021, Xiong2023}.

In the literature, there are two commonly used comparator sequences.
One is the offline optimal dynamic decision sequence $\check{x}_{[ T ]}^ *  = ( {\check{x}_1^ * , \cdots,\check{x}_T^ * } )$, where $\check{x}_t^ * \in \mathbb{X}$ is the minimizer of ${f_t}( x )$ subject to ${g_t}( x ) \le {\mathbf{0}_m}$. To guarantee that the offline optimal dynamic decision sequence $\check{x}_{[ T ]}^ *$ always exists, we assume that for any $T \in {\mathbb{N}_ + }$, the set of all the feasible decision sequences
\begin{flalign}
{\mathcal{\check{X}}_T} = \{ {( {{x_1}, \cdot  \cdot  \cdot ,{x_T}} ):{x_t} \in \mathbb{X}, {g_t}( {{x_t}} ) \le {\mathbf{0}_m},\forall t \in [ T ]} \}, \label{pl-eq5}
\end{flalign}
is non-empty. In this case, ${\rm{Net}\mbox{-}\rm{Reg}}( {\{ {{x_{i,t}}} \},{\check{x}_{[ T ]}^ *}} )$ is called the dynamic network regret. Another comparator sequence is the offline optimal static decision sequence $\hat x_{[ T ]}^ *  = ( {\hat{x}^ * , \cdots,\hat{x}^ * } )$, where $\hat{x}^ * \in \mathbb{X}$ is the minimizer of $\sum\nolimits_{t = 1}^T {{f_t}( x )} $ subject to ${g_t}( x ) \le {\mathbf{0}_m}$ for all $t \in [T]$. To guarantee that the offline optimal static decision sequence always exists, we assume that for any $T \in {N_ + }$, the set of all the feasible static decision sequences
\begin{flalign}
{\mathcal{\hat{X}}_T} = \{ {( {x, \cdot  \cdot  \cdot ,x} ):x \in \mathbb{X}, {g_t}( x ) \le {\mathbf{0}_m},\forall t \in [ T ]} \}, \label{pl-eq6}
\end{flalign}
is non-empty. In this case, ${\rm{Net}\mbox{-}\rm{Reg}}( {\{ {{x_{i,t}}} \},{\hat x_{[ T ]}^*}} )$ is called the static network regret.

In this paper, the following assumptions are made, which are commonly adopted in distributed online convex optimization, see \cite{Cao2021, Paul2022, Oakamoto2023, Xiong2023, Yi2023}, and recent survey paper \cite{Li2023} and references therein.
\begin{assumption}
(i) The set $\mathbb{X}$ is convex and closed.
Moreover, the convex set $\mathbb{X}$ contains the ball of radius $r( \mathbb{X} )$ and is contained in the ball of radius $R( \mathbb{X} )$, i.e.,
\begin{flalign}
r( \mathbb{X} ){\mathbb{B}^p} \subseteq \mathbb{X} \subseteq R( \mathbb{X} ){\mathbb{B}^p}. \label{ass-eq1}
\end{flalign}
(ii) For all $i \in [n]$, $t \in {\mathbb{N}_ + }$, the local loss functions ${f_{i,t}}$ and constraint functions ${g_{i,t}}$ are convex, and there exists a constant ${F_1}$ such that
\begin{subequations}
\begin{flalign}
| {{f_{i,t}}( x ) - {f_{i,t}}( y )} | &\le {F_1}, \label{ass-eq2a}\\
\| {{g_{i,t}}( x )} \| &\le {F_1}, x, y \in \mathbb{X}. \label{ass-eq2b}
\end{flalign}
\end{subequations}
(iii) For all $i \in [n]$, $t \in {\mathbb{N}_ + }$, the subgradients $\partial {f_{i,t}}( x )$ and $\partial {g_{i,t}}( x )$ exist, and there exists a constant ${F_2}$ such that
\begin{subequations}
\begin{flalign}
\| {\partial {f_{i,t}}( x )} \| &\le {F_2}, \label{ass-eq3a}\\
\| {\partial {g_{i,t}}( x )} \| &\le {F_2}, x \in \mathbb{X}. \label{ass-eq3b}
\end{flalign}
\end{subequations}
\end{assumption}
\begin{assumption}
For $t \in {\mathbb{N}_ + }$, the time-varying directed graph $\mathcal{G}_t$ satisfies that

\noindent (i) There exists a constant $w  \in ( {0,1} )$ such that ${[ {{W_t}} ]_{ij}} \ge w$ if $( {j,i} ) \in {\mathcal{E}_t}$ or $i = j$, and ${[ {{W_t}} ]_{ij}} = 0$ otherwise.

\noindent (ii) The mixing matrix ${W_t}$ is doubly stochastic, i.e., ${\sum\nolimits_{i = 1}^n {[ {{W_t}} ]} _{ij}} = {\sum\nolimits_{j = 1}^n {[ {{W_t}} ]} _{ij}} = 1$, $\forall i,j \in [ n ]$.

\noindent (iii) There exists an integer $B > 0$ such that the time-varying directed graph $( {\mathcal{V}, \cup _{l = 0}^{B - 1}{\mathcal{E} _{t + l}}} )$ is strongly connected.
\end{assumption}

Assumption~1 implies that the local loss functions ${f_{i,t}}$ and constraint functions ${g_{i,t}}$ are convex and Lipschitz continuous on $\mathbb{X}$.
Assumption~2 implies that the time-varying directed graph~$\mathcal{G}_t$ need not be connected at each iteration.

The considered problem is studied in \cite{Yi2023}, where the authors propose a distributed online primal--dual algorithm with two-point bandit feedback. Note that the algorithm requires that each agent broadcasts the current decision to its neighbors through the communication network at each iteration.
However, communication resources are limited that frequent communication between the agents may cause network congestion
in many practical applications, e.g., sensor networks comprised of cheap sensors with small battery capacity \cite{Cao2021}. To better utilize the limited communication resources, this paper proposes a distributed event-triggered online primal–dual algorithm with two-point bandit feedback by integrating event-triggered communication with the algorithm in \cite{Yi2023}, and establishes network regret and cumulative constraint violation bounds for the proposed algorithm. Based on these bounds, this paper also discusses the impact of event-triggered threshold on the algorithm performance.

\section{DISTRIBUTED EVENT-TRIGGERED ONLINE PRIMAL--DUAL ALGORITHM WITH TWO-POINT BANDIT FEEDBACK}

This section proposes the distributed event-triggered online primal--dual algorithm with two-point bandit feedback. The proposed algorithm can be viewed as an event-triggered version of the distributed online algorithm with two-point bandit feedback in \cite{Yi2023}, or a bandit version of the distributed event-triggered online algorithm with full-information feedback in \cite{Zhang2024}. This section also establishes network regret and cumulative constraint violation bounds for the proposed algorithm.

\subsection{Algorithm Description}
\begin{algorithm}
  \caption{Distributed Event-Triggered Online Primal--Dual Algorithm with Two-Point Bandit Feedback} 
  \begin{algorithmic}
  \renewcommand{\algorithmicrequire}{\textbf{Input:}}
  \REQUIRE
    constant ${r( \mathbb{X} )}$, non-increasing sequences $\{ {\alpha _t}\} \subseteq ( {0, + \infty })$, $\{ {\beta _t}\} \subseteq ( {0, + \infty })$, $\{ {\gamma _t}\} \subseteq ( {0, + \infty })$, $\{ {\tau _t}\} \subseteq ( {0, + \infty })$, $\{ {{\xi _t}} \} \subseteq ( {0,1} )$, and $\{ {{\delta _t}} \} \subseteq ( {0,r( \mathbb{X} ){\xi _t}} ]$.
  \renewcommand{\algorithmicrequire}{\textbf{Initialize:}}
  \REQUIRE
       ${x_{i,1}} \in {(1-\xi_{1})\mathbb{X}}$, ${{\hat x}_{i,1}} = {x_{i,1}}$ and ${q_{i,1}} = {\mathbf{0}_{{m_i}}}$.
    \STATE Broadcast ${{\hat x}_{i,1}}$ to $\mathcal{N}_i^{\text{out}}( {{\mathcal{G}_1}} )$ and receive ${{\hat x}_{j,1}}$ from $j \in \mathcal{N}_i^{\text{in}}( {{\mathcal{G}_1}} )$ for $i \in [ n ]$.
    \FOR {$t = 1, \cdot  \cdot  \cdot, T-1 $}
    \FOR {$i = 1,\cdot  \cdot  \cdot,n$ in parallel}
    \STATE Select vector ${u_{i,t}} \in {\mathbb{S}^p}$ independently and uniformly at random.
    \STATE Observe ${{f_{i,t}}( {{x_{i,t}}} )}$, ${{f_{i,t}}( {{x_{i,t}} + {\delta _t}{u_{i,t}}} )}$, ${{{[ {{g_{i,t}}( {{x_{i,t}}} )} ]}_ + }}$, and ${{{[ {{g_{i,t}}( {{x_{i,t}} + {\delta _t}{u_{i,t}}} )} ]}_ + }}$.
    \STATE \textbf{Distributed consensus protocol}:\par\nobreak\vspace{-10pt}
    \begin{small}
     \begin{flalign}
       {z_{i,t + 1}} &= \sum\limits_{j = 1}^n {{{[{W_t}]}_{ij}}{{\hat x}_{j,t}}}. \label{Algorithm1-eq1}
    \end{flalign}
      \end{small}%
    \STATE \textbf{Primal--dual protocol}:\par\nobreak\vspace{-10pt}
    \begin{small}
    \begin{subequations}
     \begin{flalign}
       {\hat{\omega} _{i,t + 1}} &= \hat{\partial} {f_{i,t}}({x_{i,t}}) + \hat{\partial} {[{g_{i,t}}({x_{i,t}})]_ + }{q_{i,t}}, \label{Algorithm1-eq2}\\
       {x_{i,t + 1}} &= {\mathcal{P}_{(1-\xi_{t+1})\mathbb{X}}}({z_{i,t + 1}} - {\alpha _{t + 1}}{\hat{\omega} _{i,t + 1}}), \label{Algorithm1-eq3}\\
       \nonumber
       {q_{i,t + 1}} &= \Big[ ( {1 - {\beta _{t + 1}}{\gamma _{t + 1}}} ){q_{i,t}} + {\gamma _{t + 1}}\Big( {[{g_{i,t}}(x_{i,t})]}_ +   \\
       &\;\;+ {{\big( {\hat{\partial} {{[ {{g_{i,t}}({x_{i,t}})} ]}_ + }} \big)}^T}( {{x_{i,t + 1}} - {x_{i,t}}} ) \Big) \Big]_ +. \label{Algorithm1-eq4}
      \end{flalign}
      \end{subequations}
      \end{small}%
      \STATE \textbf{Event-triggering check}:
     \IF {$\| {{x_{i,t + 1}} - {{\hat x}_{i,t}}} \| \ge {\tau _{t + 1}}$}
     \STATE Set ${{\hat x}_{i,t + 1}} = {x_{i,t + 1}}$, and broadcast ${{{\hat x}_{i,t + 1}}}$ to $\mathcal{N}_i^{\text{out}}( {{\mathcal{G}_{t+1}}} )$.
     \ELSE
     \STATE Set ${{\hat x}_{i,t + 1}} = {{\hat x}_{i,t}}$, and do not broadcast.
     \ENDIF
    \ENDFOR
    \ENDFOR
  \renewcommand{\algorithmicensure}{\textbf{Output:}}
  \ENSURE
      $\{ x_{i,t} \}$.
  \end{algorithmic}
\end{algorithm}
This proposed algorithm is presented in pseudo-code as Algorithm~1, and its architecture is shown in Fig.~1.
For $t \in [ T ]$ with $t \ge 2$ and ${i \in [ n ]}$,
same as the distributed online algorithm with two-point bandit feedback in \cite{Yi2023}, Algorithm~1 uses the distributed consensus protocol \eqref{Algorithm1-eq1} to compute ${z_{i,t}} \in \mathbb{X}$ for agent $i$ via the time-varying directed graph $\mathcal{G}_t$, which estimates the average value of the local decisions of all agents $\frac{1}{n}\sum\nolimits_{i = 1}^n {{x_{i,t}}} $. Then, Algorithm~1 uses the primal--dual protocol \eqref{Algorithm1-eq2}--\eqref{Algorithm1-eq4} to update the local primal variable ${x_{i,t}} \in \mathbb{X}$ and dual variable ${q_{i,t}} \in \mathbb{R}_ + ^{{m_i}}$, where ${\hat{\omega} _{i,t}}$ is the updating direction of the local primal variable ${x_{i,t}}$, ${\alpha _t}$ and ${\beta _t}$ are the step-sizes of the local primal variable ${x_{i,t}}$ and the local dual variable ${q_{i,t}}$, respectively, and ${\gamma _t}$ is the regularization parameter used to influence the structure of the local decisions.
Different from the distributed online algorithm with two-point bandit feedback in \cite{Yi2023},
Algorithm~1 uses the event-triggering check such that agent $i$ broadcasts its current local decision ${x_{i,t}}$ to its neighbors only if the norm of the difference between the current decision and its last broadcasted decision ${x_{i,t-1}}$ is not less than the current event-triggering threshold ${\tau _t}$.

Note that different from the distributed event-triggered online algorithm with full-information feedback in \cite{Zhang2024},
Algorithm~1 uses the values of the local loss function ${f_{i,t}}$ at ${x_{i,t}}$ and ${x_{i,t}} + {\delta _t}{u_{i,t}}$ to estimate the subgradient $\partial {f_{i,t}}( {{x_{i,t}}} )$, and uses the values of the local constraint function ${g_{i,t}}$ at ${x_{i,t}}$ and ${x_{i,t}} + {\delta _t}{u_{i,t}}$ to estimate the subgradient $\partial {[ {{g_{i,t}}( {{x_{i,t}}} )} ]_ + }$ as the subgradients are unavailable in bandit setting.
The subgradient approximations follow the two-point stochastic subgradient estimator proposed in \cite{Agarwal2010}, which are given by
\begin{flalign}
\nonumber
\hat \partial {f_{i,t}}( {{x_{i,t}}} ) &= \frac{p}{{{\delta _t}}}\big( {{f_{i,t}}( {{x_{i,t}} + {\delta _t}{u_{i,t}}} ) - {f_{i,t}}( {{x_{i,t}}} )} \big){u_{i,t}} \in {\mathbb{R}^p}, \\
\nonumber
\hat \partial {[ {{g_{i,t}}( {{x_{i,t}}} )} ]_ + } &= \frac{p}{{{\delta _t}}}{\big( {{{[ {{g_{i,t}}( {{x_{i,t}} + {\delta _t}{u_{i,t}}} )} ]}_ + } - {{[ {{g_{i,t}}( {{x_{i,t}}} )} ]}_ + }} \big)^T}
\otimes {u_{i,t}} \in {\mathbb{R}^{p \times {m_i}}},
\end{flalign}
where ${\delta _t} \in ( {0,r( \mathbb{X} ){\xi _t}} ]$ is an exploration parameter, ${r( \mathbb{X} )}$ is a constant given in Assumption~1, ${\xi _t} \in ( {0,1} )$ is a shrinkage coefficient, and ${u_{i,t}} \in {\mathbb{S}^p}$ is a uniformly distributed random vector.

\begin{figure*}
 \centering
  \includegraphics[width=16cm]{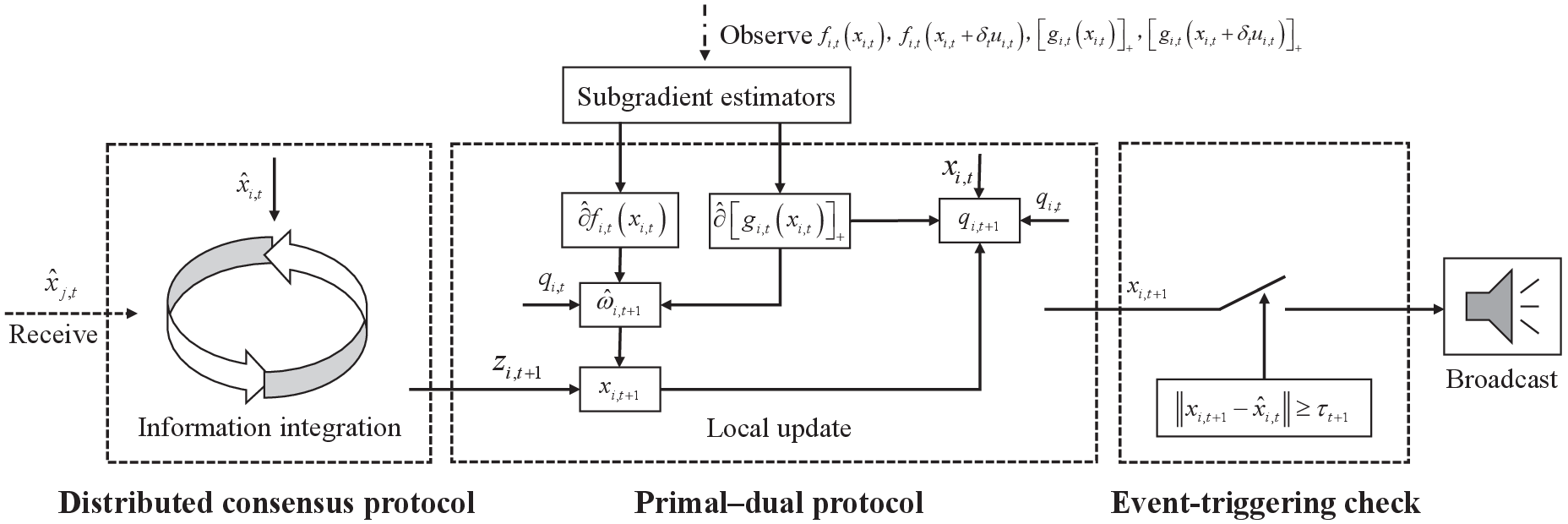}
  \caption{Architecture of the distributed event-triggered online primal--dual algorithm with two-point bandit feedback.}
\end{figure*}

\subsection{Performance Analysis}
We first appropriately design the parameter sequences for Algorithm~1, and establish dynamic network regret and network cumulative constraint violation bounds in the following theorem.
\begin{theorem}\label{thm3}
Suppose Assumptions 1 and 2 hold. Let $\{ {{x_{i,t}}} \}$ be the sequences generated by Algorithm~1 with
\begin{flalign}
{\alpha _t} = \sqrt {\frac{{{\Psi _t}}}{t}}, {\beta _t} = \frac{1}{{{t^\kappa }}}, {\gamma _t} = \frac{1}{{{t^{1 - \kappa }}}},
{\xi _t} = \frac{1}{{t + 1}}, {\delta _t} = \frac{{r( \mathbb{X} )}}{{t + 1}}, \forall t \in {\mathbb{N}_ + }, \label{theorem1-eq1}
\end{flalign}
where ${\Psi _t} = \sum\nolimits_{k = 1}^t {{\tau _k}} $, $\kappa  \in ( {0,1} )$ are constants. Then, for any $T \in {\mathbb{N}_ + }$ and any comparator sequence ${y_{[ T ]}} \in {\mathcal{X}_T}$,
\begin{flalign}
\mathbf{E}[{{\rm{Net}\mbox{-}\rm{Reg}}( {\{ {{x_{i,t}}} \},{y_{[ T ]}}} )}] &= \mathcal{O}( {T^\kappa } + \sqrt {{\Psi_T} T}  + \sqrt {{{\Psi_T} ^{ - 1}}T} {P_T} ), \label{theorem1-eq2}\\
\mathbf{E}[{{\rm{Net} \mbox{-} \rm{CCV}}( {\{ {{x_{i,t}}} \}} )}]
&= \mathcal{O}( {T^{1 - \kappa /2}} + \Psi _T^{1/4}{T^{3/4}} ). \label{theorem1-eq3}
\end{flalign}
where ${P_T} = \sum\nolimits_{t = 1}^{T - 1} {\| {{y_{t + 1}} - {y_t}} \|} $ is the path-length of the comparator sequence ${y_{[ T ]}}$.
\end{theorem}
\begin{proof}
The proof is given in Appendix B.
\end{proof}

\begin{remark}\label{rem1}
Sublinear dynamic network regret bound \eqref{theorem1-eq2} and network cumulative constraint violation bound \eqref{theorem1-eq3} would be established if the path-length of the comparator sequence grows sublinearly and the event-triggering threshold $\tau_t$ converges to zero, i.e., $\sum\nolimits_{k = 1}^t {{\tau _k}}$ grows sublinearly.
The bounds provided in \eqref{theorem1-eq2} and \eqref{theorem1-eq3} characterize the impact of event-triggered threshold $\tau_t$ on dynamic network regret and network cumulative constraint violation through ${\Psi_T}$.
The larger the event-triggering threshold $\tau_t$, the larger the static part of the bound \eqref{theorem1-eq2} (i.e., ${T^\kappa } + \sqrt {{\Psi_T} T}$) and the bound \eqref{theorem1-eq3}, and the smaller the dynamic part (i.e., $\sqrt {{{\Psi_T} ^{ - 1}}T} {P_T}$) of the bound \eqref{theorem1-eq2}.
The dynamic network regret bound \eqref{theorem1-eq2} and network cumulative constraint violation bound \eqref{theorem1-eq3} are the same as the results achieved by the distributed event-triggered online algorithm with full-information feedback in \cite{Zhang2024}.
When event-triggered communication is not considered, i.e., ${\tau _1} = 1$, ${\tau _t} = 0$ for $t \in [ {2,T} ]$, these bounds recover the results achieved by the centralized online algorithm with two-point bandit feedback in \cite{Cao2019}.
When inequality constraints are not considered, i.e., given any $x \in \mathbb{X}$, ${g_{i,t}}( x ) \equiv \mathbf{0}_{m_i}$ for $i \in [ n ]$ and $t \in [ T ]$, the dynamic network regret bound \eqref{theorem1-eq2} would recover the results achieved by the distributed event-triggered online algorithms with two-point bandit feedback in \cite{Cao2021, Xiong2023} if the path-length of the comparator sequence is zero for any~$T$, i.e., ${P_T} \equiv 0$ for any $T$.
When both event-triggered communication and inequality constraints are not considered, the dynamic network regret bound \eqref{theorem1-eq2} would recover the results achieved by the centralized online algorithm with two-point bandit feedback in \cite{Agarwal2010} and the distributed online algorithm with two-point bandit feedback in \cite{Yuan2021} if the path-length of the comparator sequence is zero for any~$T$.
Note that we analyze dynamic regret, while the authors of \cite{Agarwal2010, Yuan2021, Cao2021, Xiong2023} analyze static regret. However, When event-triggered communication is not considered, the network cumulative constraint violation bound \eqref{theorem1-eq3} is larger than that achieved by the centralized online algorithm with two-point bandit feedback in \cite{Chen2019}. This is reasonable since the authors of \cite{Chen2019} consider Slater’s condition for constraint functions (i.e., there is a point that strictly satisfies inequality constraints), which is a sufficient condition for strong duality to hold \cite{Boyd2004}, moreover, the algorithm in \cite{Chen2019} uses full-information feedback for constraint functions. In addition, the dynamic network regret bound \eqref{theorem1-eq2} and network cumulative constraint violation bound \eqref{theorem1-eq3} do not recover the results achieved by the distributed online algorithm with two-point bandit feedback in \cite{Yi2023} since the step-size $\alpha_t$ of the local primal variable is only a special case of that in \cite{Yi2023}. If choosing ${\alpha _t} = 1/\sqrt t$ for the algorithm in \cite{Yi2023}, these bounds would recover the results achieved by the algorithm in \cite{Yi2023}.
\end{remark}

\begin{remark}\label{rem2}
The proof of Theorem~1 has substantial differences compared to that of Theorem~3 in~\cite{Yi2023}.
More specifically, in our Algorithm~1, the agents would not broadcast the current local decisions if the event-triggering condition is not satisfied.
The resulting local decision sequences are different with those produced by the distributed online algorithm without event-triggered communication in~\cite{Yi2023} although the updating rules are similar.
This critical difference leads to challenges in theoretical proof because we need to reanalyse the local decisions at each iteration.
To tackle these challenges, we analyse the difference between the produced local decision ${x_{i,t}}$ and the stored local decision ${{\hat x}_{i,t}}$ for running distributed consensus protocol at each iteration for agent~$i$, $i \in [ n ]$ in the proof of Lemma~4. The analysis shows that the norm of the difference can be bounded by the current event-triggering threshold regardless of whether the event-triggering condition is satisfied or not.
The dynamic network regret bound \eqref{theorem1-eq2} and network cumulative constraint violation bound \eqref{theorem1-eq3} are established in this way, and are thus subject to event-triggering threshold.
In addition, the proof of Theorem~1 has significant differences compared to that of Theorem~1 in \cite{Zhang2024}.
Note that the subgradients of local loss and constraint functions are unavailable in bandit setting.
Our Algorithm~1 uses two-point stochastic subgradient estimators for the subgradients.
However, there exist gaps between the estimators and the real subgradients as the estimators are unbiased subgradients of the uniformly smoothed versions of local loss and constraint functions.
By utilizing the property of local loss and constraint functions, e.g., boundedness, convexity and Lipschitz continuity, the relationship between the smooth functions and their original functions can be established, see Lemma~1. Based on Lemma~1, we reanalyse dynamic network regret and network cumulative constraint violation bounds.
\end{remark}

Then, we consider two classes of explicit expressions for the event-triggering threshold $\tau_t$. Firstly, we select the event-triggering threshold sequence produced by ${\tau _t} = 1/{t^\theta }$, which is also adopted by the distributed online algorithms in \cite{Cao2021, Paul2022, Xiong2023, Oakamoto2023}. We establish dynamic network regret and network cumulative constraint violation bounds in the following corollary.
\begin{corollary}\label{cor1}
Under the same conditions as in Theorem 1 with ${\tau _t} = 1/{t^\theta }$ and $\theta  > 0$, for any $T \in {\mathbb{N}_ + }$ and any comparator sequence ${y_{[ T ]}} \in {\mathcal{X}_T}$, it holds that
\begin{flalign}
\mathbf{E}[{{\rm{Net}\mbox{-}\rm{Reg}}( {\{ {{x_{i,t}}} \},{y_{[ T ]}}} )}]
&= \left\{ \begin{array}{l}
\mathcal{O}( {{T^{\max \{ {\kappa ,1 - \theta /2} \}}} + {T^{\theta /2}}{P_T}} ), \;\;\;\;\;\;\;\;\;\;\;\mathrm{if} \; 0 < \theta < 1, \\
\mathcal{O}\big( {{T^\kappa } + \sqrt {T\log ( T )}  + \sqrt {\frac{T}{{\log ( T )}}} {P_T}} \big), \mathrm{if} \; \theta = 1, \\
\mathcal{O}( {{T^{\max \{ {\kappa ,1/2} \}}} + {\sqrt T }{P_T}} ), \;\;\;\;\;\;\;\;\;\;\;\;\;\;\;\:\mathrm{if} \; \theta > 1,
\end{array} \right. \label{corollary1-eq1} \\
\mathbf{E}[{{\rm{Net} \mbox{-} \rm{CCV}}( {\{ {{x_{i,t}}} \}} )}]
&= \left\{ \begin{array}{l}
\mathcal{O}( {{T^{\max \{ {1 - \kappa /2,1 - \theta /4} \}}}} ), \;\;\;\;\;\;\;\;\;\;\;\;\;\;\;\;\;\;\;\;\;\mathrm{if} \; 0 < \theta < 1, \\
\mathcal{O}( {{T^{1 - \kappa /2}} + {T^{3/4}}\log {(T)^{1/4}}} ), \;\;\;\;\;\;\;\;\;\,\,\mathrm{if} \; \theta = 1, \\
\mathcal{O}( {{T^{\max \{ {1 - \kappa /2,3/4} \}}}} ), \;\;\;\;\;\;\;\;\;\;\;\;\;\;\;\;\;\;\;\;\;\;\;\:\,\mathrm{if} \; \theta > 1.
\end{array} \right. \label{corollary1-eq2}
\end{flalign}
\end{corollary}

\begin{remark}\label{rem3}
The dynamic network regret bound \eqref{corollary1-eq1} and network cumulative constraint violation bound \eqref{corollary1-eq2} would be sublinear if the path-length ${P_T}$ grows sublinearly. Moreover, the larger $\theta$, the smaller the static part of the bound \eqref{corollary1-eq1} and the bound \eqref{corollary1-eq2}, and the larger the dynamic part of the bound \eqref{corollary1-eq1}. When $\theta  > 1$, these bounds recover the results achieved by the centralized online algorithm with two-point bandit feedback in~\cite{Cao2019}. In addition, when inequality constraints are not considered, the dynamic network regret bound~\eqref{corollary1-eq1} would recover the results achieved by the centralized online algorithm with two-point bandit feedback in \cite{Agarwal2010}, the distributed online algorithm with two-point bandit feedback in \cite{Yuan2021}, and the distributed event-triggered online algorithm with two-point bandit feedback in \cite{Cao2021, Xiong2023} if the path-length of the comparator sequence is zero for any~$T$.
In addition, we consider time-varying inequality constraints, while the authors of \cite{Agarwal2010, Cao2021, Xiong2023} do not consider inequality constraints.
\end{remark}

Secondly, we choose the event-triggering threshold sequence produced by ${\tau _t} =1/{c^t}$, which is also adopted in distributed optimization with event-triggered communication, see, e.g., \cite{Seyboth2013, Yang2016b, Ding2017, Ge2020, Yang2022}. We establish dynamic network regret and network cumulative constraint violation bounds in the following corollary.
\begin{corollary}\label{cor2}
Under the same conditions as in Theorem 1 with ${\tau _t} = 1/{c^t}$ and $c  > 1$, for any $T \in {\mathbb{N}_ + }$ and any comparator sequence ${y_{[ T ]}} \in {\mathcal{X}_T}$, it holds that
\begin{flalign}
\mathbf{E}[{{\rm{Net}\mbox{-}\rm{Reg}}( {\{ {{x_{i,t}}} \},{y_{[ T ]}}} )}] &= \mathcal{O}( {{T^{\max \{ {\kappa ,1/2} \}}} + {\sqrt T }{P_T}} ), \label{corollary2-eq1}\\
\mathbf{E}[{{\rm{Net} \mbox{-} \rm{CCV}}( {\{ {{x_{i,t}}} \}} )}] &= \mathcal{O}( {{T^{\max \{ {1 - \kappa /2,3/4} \}}}} ). \label{corollary2-eq2}
\end{flalign}
\end{corollary}

\begin{remark}\label{rem4}
The dynamic network regret bound \eqref{corollary2-eq1} and network cumulative constraint violation bound \eqref{corollary2-eq2} recover the results in Corollary 1 with $\theta  > 1$ and achieved by the centralized online algorithm with two-point bandit feedback in \cite{Cao2019}. Moreover, when inequality constraints are not considered, the dynamic network regret bound \eqref{corollary2-eq1} recovers the results achieved by the centralized online algorithm with two-point bandit feedback in \cite{Agarwal2010}, the distributed online algorithm with two-point bandit feedback in \cite{Yuan2021}, and the distributed event-triggered online algorithm with two-point bandit feedback in \cite{Cao2021, Xiong2023} if the path-length of the comparator sequence is zero for any~$T$.
\end{remark}

Note that the event-triggering threshold $\tau_t$ in \eqref{theorem1-eq1} affects the step-size $\alpha_t$ of the local primal variable. To avoid that, we next show how to independently design $\alpha_t$ in the following theorem.

\begin{theorem}\label{thm4}
Suppose Assumptions 1 and 2 hold. Let $\{ {{x_{i,t}}} \}$ be the sequences generated by Algorithm~1 with
\begin{flalign}
{\alpha _t} = \frac{{{\alpha _0}}}{{{t^{{\theta _1}}}}}, {\beta _t} = \frac{1}{{{t^{{\theta _2}}}}}, {\gamma _t} = \frac{1}{{{t^{1 - {\theta _2}}}}},
{\xi _t} = \frac{1}{{t + 1}}, {\delta _t} = \frac{{r( \mathbb{X} )}}{{t + 1}}, {\tau _t} = \frac{{{\tau _0}}}{{{t^{{\theta _3}}}}}, \forall t \in {\mathbb{N}_ + }, \label{theorem2-eq1}
\end{flalign}
where ${\alpha _0}$, ${\theta _1} \in ( {0,1} )$, ${\theta _2}  \in ( {0,1} )$, and ${\theta _3}$ are positive constants, and  ${\tau _0}$ is a non-negative constant. Then, for any $T \in {\mathbb{N}_ + }$ and any comparator sequence ${y_{[ T ]}} \in {\mathcal{X}_T}$,
\begin{flalign}
\mathbf{E}[{{\rm{Net}\mbox{-}\rm{Reg}}( {\{ {{x_{i,t}}} \},{y_{[ T ]}}} )}]
&= \left\{ \begin{array}{l}
\mathcal{O}\big( {\alpha _0}{T^{1 - {\theta _1}}} + {T^{{\theta _2}}} + \frac{{{\tau _0}}}{{{\alpha _0}}}{T^{1 + {\theta _1} - {\theta _3}}}\\
\;\;\;\;\;\;\;\;\;\;\;\;\;\;\;\;\;\;\; + \frac{{{T^{{\theta _1}}}( {1 + {P_T}} )}}{{{\alpha _0}}} \big), \;\;\;\;\;\;\mathrm{if} \; {\theta _1} < {\theta _3} < 1 + {\theta _1},\\
\mathcal{O}\big( {\alpha _0}{T^{1 - {\theta _1}}} + {T^{{\theta _2}}} + \frac{{{\tau _0}}}{{{\alpha _0}}}\log ( T ) \\
\;\;\;\;\;\;\;\;\;\;\;\;\;\;\;\;\;\;\; + \frac{{{T^{{\theta _1}}}( {1 + {P_T}} )}}{{{\alpha _0}}} \big),\;\;\;\;\;\;\mathrm{if} \; {\theta _3} = 1 + {\theta _1},\\
\mathcal{O}\big( {\alpha _0}{T^{1 - {\theta _1}}} + {T^{{\theta _2}}} + \frac{{{\tau _0}}}{{{\alpha _0}}} \\
\;\;\;\;\;\;\;\;\;\;\;\;\;\;\;\;\;\;\; + \frac{{{T^{{\theta _1}}}( {1 + {P_T}} )}}{{{\alpha _0}}} \big),\;\;\;\;\;\;\mathrm{if} \; {\theta _3} > 1 + {\theta _1},
\end{array} \right. \label{theorem2-eq2} \\
\mathbf{E}[{{\rm{Net} \mbox{-} \rm{CCV}}( {\{ {{x_{i,t}}} \}} )}]
&=\left\{ \begin{array}{l}
\mathcal{O}( \sqrt {{\alpha _0}} {T^{1 - {\theta _1}/2}} + {T^{1 - {\theta _2}/2}}\\
\;\;\;\;\;\;\;\;\;\;\;\;\;\;\;\;\;\;\; + \sqrt {{\tau _0}} {T^{1 - {\theta _3}/2}} ), \;\;\,\, \mathrm{if} \; {\theta _1} < {\theta _3} < 1,\\
\mathcal{O}\big( \sqrt {{\alpha _0}} {T^{1 - {\theta _1}/2}} + {T^{1 - {\theta _2}/2}}\\
\;\;\;\;\;\;\;\;\;\;\;\;\;\;\;\;\;\;\; + \sqrt {{\tau _0}T\log ( T )} \big),\,\mathrm{if} \; {\theta _3} = 1,\\
\mathcal{O}( \sqrt {{\alpha _0}} {T^{1 - {\theta _1}/2}} + {T^{1 - {\theta _2}/2}} \\
\;\;\;\;\;\;\;\;\;\;\;\;\;\;\;\;\;\;\; + \sqrt {{\tau _0 }{T}} ), \;\;\;\;\;\;\;\;\;\;\;\,\, \mathrm{if} \; {\theta _3} > 1.
\end{array} \right. \label{theorem2-eq3}
\end{flalign}
\end{theorem}
\begin{proof}
The proof is given in Appendix C.
\end{proof}

\begin{remark}\label{rem5}
The dynamic network regret bound \eqref{theorem2-eq2} and network cumulative constraint violation bound \eqref{theorem2-eq3} are the same as the results achieved by the distributed event-triggered online algorithm with full-information feedback in \cite{Zhang2024}, that is, in an average sense, our Algorithm~1 is as efficient as its full-information feedback version. When event-triggered communication is not considered, i.e., ${\tau _0} = 0$, these bounds would recover the results achieved by the distributed online algorithm with two-point bandit feedback in \cite{Yi2023} if ${\theta _1} = {\theta _2}$.
\end{remark}

\begin{remark}\label{rem6}
By replacing the comparator sequence ${y_{[ T ]}}$ with the offline optimal static decision sequence $\hat x_{[ T ]}^*$, we have ${P_T} \equiv 0$ for any $T$, and then the static network regret and cumulative constraint violation bounds for Algorithm~1 with corresponding parameter and event-triggered threshold sequences can be easily established based on the results in Theorems~1, ~2, Corollaries~1 and~2, respectively.
These bounds are the same as \eqref{theorem1-eq2}--\eqref{corollary2-eq2}, \eqref{theorem2-eq2}, and \eqref{theorem2-eq3} with ${P_T} = 0$, respectively.
When inequality constraints are not considered, these static network regret bounds recover the results achieved by the distributed event-triggered online algorithms with two-point bandit feedback in \cite{Cao2021, Xiong2023}.
When both event-triggered communication and inequality constraints are not considered, these static network regret bounds recover the results achieved by the centralized online algorithm with two-point bandit feedback in \cite{Agarwal2010} and the distributed online algorithm with two-point bandit feedback in \cite{Yuan2021}.
\end{remark}

\section{NUMERICAL EXAMPLE}
To evaluate the performance of Algorithm~1, we consider a distributed online linear regression problem with time-varying linear inequality constraints over a network of $n$ agents.
At iteration~$t$, the local loss and constraint functions are  ${f_{i,t}}( x ) = \frac{1}{2}{( {{A_{i,t}}x - {\vartheta _{i,t}}} )^2}$ and ${g_{i,t}}( x ) = {B_{i,t}}x - {b_{i,t}}$, respectively, where each component of ${A_{i,t}} \in {\mathbb{R}^{{q_i} \times p}}$ is randomly generated from the uniform distribution in the interval $[ { - 1,1} ]$, ${\vartheta _{i,t}} = {A_{i,t}}{\mathbf{1}_p} + {\zeta _{i,t}}$ with ${\vartheta _{i,t}} \in {\mathbb{R}^{{q_i}}}$ and ${\zeta _{i,t}}$ being a standard normal random vector, and each component of ${B_{i,t}} \in {\mathbb{R}^{{m_i} \times p}}$ and ${b_{i,t}} \in {\mathbb{R}^{{m_i}}}$ is randomly generated from the uniform distribution in the interval $[ {0,2} ]$ and $[ {0,1} ]$, respectively.
We set $n = 100$, ${q_i} = 4$, $p = 10$, ${m_i} = 2$, $\mathbb{X} = {[ { - 5,5} ]^p}$.
We use a time-varying undirected graph to model the communication topology.
Specifically, at each iteration~$t$, the graph is first randomly generated where the probability of any two agents being connected is $0.1$. Then, to make sure that Assumption~2 is satisfied, we add edges~$( {i,i + 1} )$ for $i = 1, \cdot  \cdot  \cdot ,24$ when $t \in \{ {4c + 1} \}$, edges~$( {i,i + 1} )$ for $i = 25, \cdot  \cdot  \cdot ,49$ when $t \in \{ {4c + 2} \}$, edges~$( {i,i + 1} )$ for $i = 50, \cdot  \cdot  \cdot ,74$ when $t \in \{ {4c + 3} \}$, edges~$( {i,i + 1} )$ for $i = 75, \cdot  \cdot  \cdot ,99$ when $t \in \{ {4c + 4} \}$ with $c$ being a nonnegative integer. Moreover, let
${[ {{W_t}} ]_{ij}} = \frac{1}{n}$ if $( {j,i} ) \in {\mathcal{E}_t}$ and ${[ {{W_t}} ]_{ii}} = 1 - \sum\nolimits_{j = 1}^n {{{[ {{W_t}} ]}_{ij}}} $.

\begin{figure}[!ht]
 \centering
  \includegraphics[width=8cm]{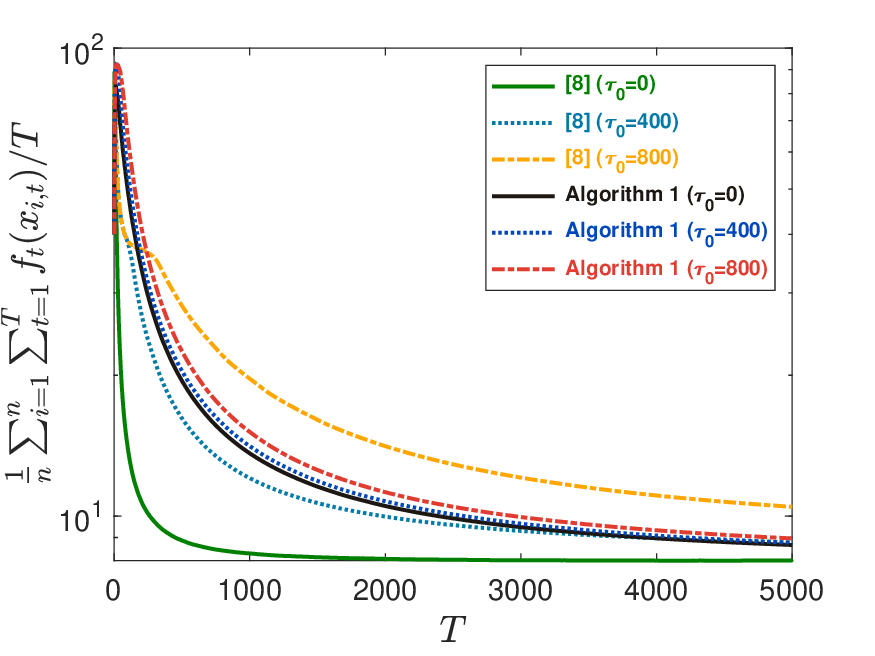}
  \caption{Evolutions of $\frac{1}{n}\sum\nolimits_{i = 1}^n {\sum\nolimits_{t = 1}^T {{f_t}( {{x_{i,t}}} )} } /T$.}
\end{figure}

\begin{figure}[!ht]
 \centering
  \includegraphics[width=8cm]{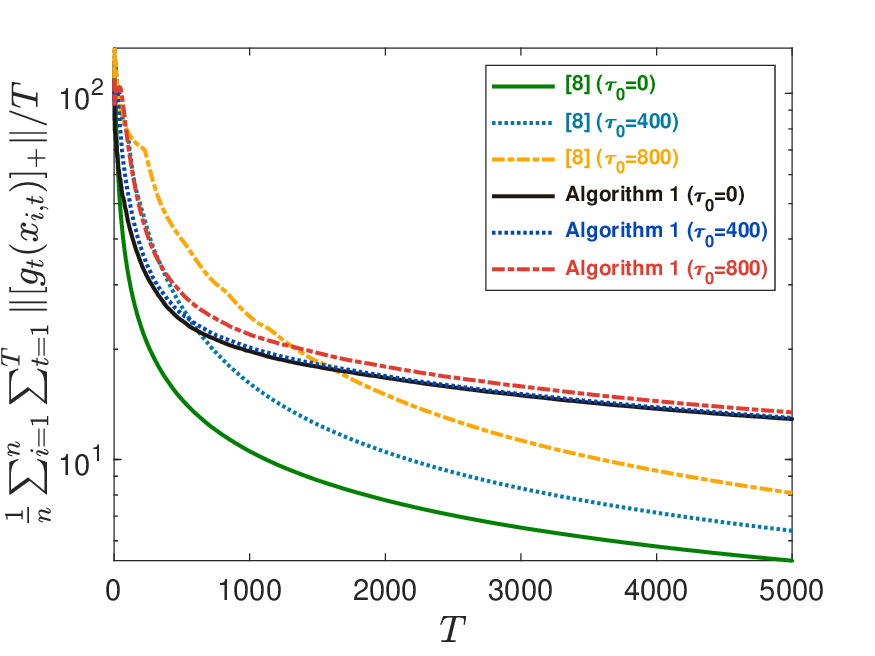}
  \caption{Evolutions of $\frac{1}{n}\sum\nolimits_{i = 1}^n {\sum\nolimits_{t = 1}^T {\| {{{[ {{g_t}( {{x_{i,t}}} )} ]}_ + }} \|} } /T$.}
\end{figure}

\begin{figure}[!ht]
 \centering
  \includegraphics[width=8cm]{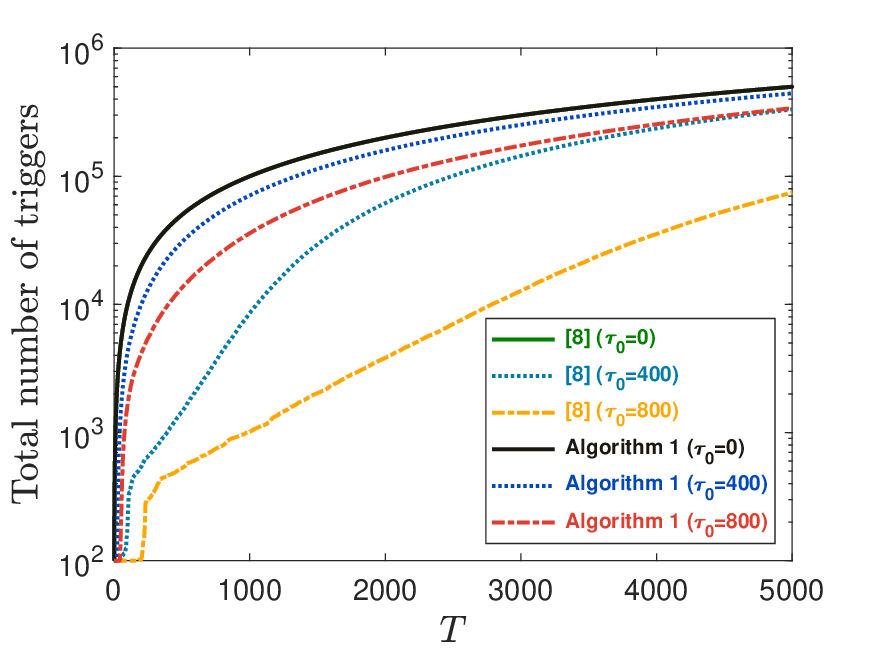}
  \caption{Evolutions of total number of triggers.}
\end{figure}
Note that there are no other distributed event-triggered online algorithms with bandit feedback to solve the considered problem due to the time-varying constraints. We compare our Algorithm~1 with the distributed event-triggered online algorithm with full-information in \cite{Zhang2024}.

Set ${\alpha _t} = 1/\sqrt t $, ${\beta _t} = 1/\sqrt t$, ${\gamma _t} = 1/\sqrt t$, and ${\tau _t} = {\tau _0}/t$ for our Algorithm~1 and the distributed event-triggered online algorithm in \cite{Zhang2024}. To explore the impact of different event-triggering threshold sequences on network regret and cumulative constraint violation, we select ${\tau _0} = 0$, ${\tau _0} = 400$, and ${\tau _0} = 800$, respectively.
With different values of ${\tau _0}$, Figs.~2--4 illustrate the evolutions of the average cumulative loss $\frac{1}{n}\sum\nolimits_{i = 1}^n {\sum\nolimits_{t = 1}^T {{f_t}( {{x_{i,t}}} )} } /T$, the average cumulative constraint violation $\frac{1}{n}\sum\nolimits_{i = 1}^n {\sum\nolimits_{t = 1}^T {\| {{{[ {{g_t}( {{x_{i,t}}} )} ]}_ + }} \|} } /T$ and total number of triggers, respectively.
The results show that as ${\tau _0}$ increases, the average cumulative loss and the average cumulative constraint violation increase, while the total number of triggers decreases, which are consistent with Theorem~2. In addition, Figs.~2 and~3 also show that our proposed algorithm has larger average cumulative loss and constraint violation than the algorithm in \cite{Zhang2024} under the same ${\tau _0}$. However, the disadvantages gradually weaken as ${\tau _0}$ increases, and our Algorithm~1 has smaller average cumulative loss when ${\tau _0} = 800$.

\section{CONCLUSIONS}
This paper considered the distributed bandit convex optimization problem with time-varying inequality constraints.
To better utilize communication resources, we proposed the distributed event-triggered online primal--dual algorithm with two-point bandit feedback.
We analyzed the network regret and cumulative constraint violation for the proposed algorithm under several classes of appropriately chosen decreasing
parameter sequences and non-increasing event-triggered
threshold sequences.
Our theoretical results were comparable to the results achieved by distributed event-triggered online algorithms with full-information feedback.
In brief, this paper broadened the applicability of distributed event-triggered online convex optimization to the regime with time-varying constraints.

The future direction is to investigate new distributed online algorithms with two-point bandit feedback such that the network regret bound can be reduced under the strongly convex condition and the network cumulative constraint violation bound can be reduced when the global constraint function satisfies Slater’s condition. Moreover, we will also investigate compressed communication to reduce communication overhead at each iteration.

\appendix

\hspace{-3mm}\emph{A. Useful Lemmas}

Some preliminary results are given in this section.
\begin{lemma}\label{lem1}
(Lemma 8 in \cite{Yi2023}) If Assumption~1 holds. Then, ${{\hat f}_{i,t}}( x )$ and ${[ {{{\hat g}_{i,t}}( x )} ]_ + }$ are convex on $( {1 - {\xi _t}} )\mathbb{X}$, and for any $i \in [ n ]$, $t \in {\mathbb{N}_ + }$, $x \in ( {1 - {\xi _t}} )\mathbb{X}$, $q \in \mathbb{R}_ + ^{{m_i}}$,
\begin{subequations}
\begin{flalign}
&\partial {{\hat f}_{i,t}}( x ) = {\mathbf{E}_{{\mathfrak{U}_t}}}[ {\hat \partial {f_{i,t}}( x )} ], \label{lemma1-eq1a}\\
&{f_{i,t}}\left( x \right) \le {{\hat f}_{i,t}}\left( x \right) \le {f_{i,t}}\left( x \right) + {F_2}{\delta _t}, \label{lemma1-eq1b}\\
&\| {\hat \partial {f_{i,t}}( x )} \| \le p{F_2}, \label{lemma1-eq1c}\\
&\partial {[ {{{\hat g}_{i,t}}( x )} ]_ + } = {\mathbf{E}_{{\mathfrak{U}_t}}}\big[ {\hat \partial {{[ {{g_{i,t}}( x )} ]}_ + }} \big], \label{lemma1-eq1d} \\
&{q^T}{[ {{g_{i,t}}( x )} ]_ + } \le {q^T}{[ {{{\hat g}_{i,t}}( x )} ]_ + }
\le {q^T}{[ {{g_{i,t}}( x )} ]_ + } + {F_2}{\delta _t}\| q \|, \label{lemma1-eq1e}\\
&\| {\hat \partial {{[ {{g_{i,t}}( x )} ]}_ + }} \| \le p{F_2}, \label{lemma1-eq1f}\\
&\| {{{[ {{{\hat g}_{i,t}}( x )} ]}_ + }} \| \le {F_1}, \label{lemma1-eq1g}
\end{flalign}
\end{subequations}
where ${{\hat f}_{i,t}}( x ) = {\mathbf{E}_{v \in {\mathbb{B}^p}}}[ {{f_{i,t}}( {x + {\delta _t}v} )} ]$ and ${[ {{{\hat g}_{i,t}}( x )} ]_ + } = {\mathbf{E}_{v \in {\mathbb{B}^p}}}\big[ {{{[ {{g_{i,t}}( {x + {\delta _t}v} )} ]}_ + }} \big]$ with $v$ being chosen uniformly at random, and ${\mathfrak{U}_t}$ is the $\sigma$-algebra induced by the independent and identically distributed variables ${u_{1,t}}, \cdot  \cdot  \cdot ,{u_{n,t}}$.
\end{lemma}

\begin{lemma}\label{lem2}
(Lemma 4 in \cite{Yi2023})
If Assumption 2 holds. For all $i \in [ n ]$ and $t \in {\mathbb{N}_ + }$, ${{\hat x}_{i,t}}$ generated by Algorithm 1 satisfy
\begin{flalign}
\| {{{\hat x}_{i,t}} - {{\bar x}_t}} \| &\le \tau {\lambda ^{t - 2}}\sum\limits_{j = 1}^n {\| {{{\hat x}_{j,1}}} \|}  + \frac{1}{n}\sum\limits_{j = 1}^n {\| {\hat{\varepsilon} _{j,t - 1}^x} \|}  + \| {\hat{\varepsilon} _{i,t - 1}^x} \|
+ \tau \sum\limits_{s = 1}^{t - 2} {{\lambda ^{t - s - 2}}} \sum\limits_{j = 1}^n {\| {\hat{\varepsilon} _{j,s}^x} \|}, \label{lemma2-eq1}
\end{flalign}
where ${{\bar x}_t} = \frac{1}{n}\sum\nolimits_{j = 1}^n {{{\hat x}_{j,t}}} $ and $\hat{\varepsilon} _{i,t - 1}^x = {{\hat x}_{i,t}} - {z_{i,t}}$.
\end{lemma}

\begin{lemma}\label{lem3}
(Lemma 9 in \cite{Yi2023})
Suppose Assumptions 1 and 2 hold, and ${\gamma _t}{\beta _t} \le 1$, $t \in {\mathbb{N}_ + }$. For all $i \in [ n ]$ and $t \in {\mathbb{N}_ + }$, the sequences ${q_{i,t}}$ generated by Algorithm 2 satisfy
\begin{flalign}
\| {{\beta _t}{q_{i,t}}} \| \le {{\hat \varpi }_1}, \label{lemma3-eq1}
\end{flalign}
\begin{flalign}
{\Delta _{i,t}}( {{\mu _i}} ) &\le 2{{\hat \varpi }_1}^2{\gamma _t} + q_{i,t - 1}^T{{\hat b}_{i,t}} - \mu _i^T{[ {{g_{i,t - 1}}( {{x_{i,t - 1}}} )} ]_ + }
 + \frac{1}{2}{\beta _t}{\| {{\mu _i}} \|^2} + p{F_2}\| {{\mu _i}} \|\| {{x_{i,t}} - {x_{i,t - 1}}} \|, \label{lemma3-eq2}
\end{flalign}
where ${{\hat \varpi }_1} = {F_1} + 2p{F_2}R( \mathbb{X} )$, and ${\hat{b}_{i,t}} = {[ {{g_{i,t - 1}}( {{x_{i,t - 1}}} )} ]_ + } + {\big( {\hat{\partial} {{[ {{g_{i,t - 1}}( {{x_{i,t - 1}}} )} ]}_ + }} \big)^T}( {{x_{i,t}} - {x_{i,t - 1}}} )$.
\end{lemma}

Next, we present network regret bound at one iteration.

Note that since there exists the event-triggering check in our Algorithm~1, the local decisions and corresponding losses of the agents may be different with those of the distributed online algorithm with two-point bandit feedback in \cite{Yi2023} although the updating rules are similar. We give a new bound for the average of network-wide loss at one iteration based on the behavior of Algorithm~1 under event-triggering check in the proof of the following lemma, which is critical to re-derive network regret bound at one iteration.
\begin{lemma}\label{lem4}
Suppose Assumptions 1 and 2 hold. For all $i \in [ n ]$, let $\{ {{x_{i,t}}}\}$ be the sequences generated by Algorithm~2 and $\{ {{y_t}} \}$ be an arbitrary sequence in $\mathbb{X}$, then
\begin{flalign}
\nonumber
&\;\;\;\;\; \frac{1}{n}\sum\limits_{i = 1}^n {{f_t}( {{x_{i,t}}} )}  - {f_t}( {{y_t}} ) \\
\nonumber
&\le \frac{1}{n}\sum\limits_{i = 1}^n {q_{i,t}^T\big( {{{[ {{g_{i,t}}( {{y_t}} )} ]}_ + } - {\mathbf{E}_{{\mathfrak{U}_t}}}[ {{{\hat b}_{i,t + 1}}} ]} \big)}
+ \frac{1}{n}\sum\limits_{i = 1}^n {{F_2}\big( {2\| {{\hat{x}_{i,t}} - {{\bar x}_{t}}} \| + p{\mathbf{E}_{{\mathfrak{U}_t}}}[ {\| {{x_{i,t}} - {x_{i,t + 1}}} \|} ]} \big)} \\
\nonumber
&\;\; + \frac{1}{{2n{\alpha _{t + 1}}}}\sum\limits_{i = 1}^n {{\mathbf{E}_{{\mathfrak{U}_t}}}\big[ {{{\| {{{\hat y}_t} - {z_{i,t + 1}}} \|}^2} - {{\| {{{\hat y}_{t + 1}} - {z_{i,t + 2}}} \|}^2}} }
 { + {{\| {{{\hat y}_{t + 1}} - {\hat{x}_{i,t + 1}}} \|}^2} - {{\| {{{\hat y}_t} - {x_{i,t + 1}}} \|}^2}} \big] \\
&\;\; + \frac{1}{n}\sum\limits_{i = 1}^n {{F_2}} \big( {R( \mathbb{X} ){\xi _t} + {\delta _t}} \big)\big( {\| {{q_{i,t}}} \| + 1} \big) + 2{F_2}{\tau _t} - \frac{1}{n}\sum\limits_{i = 1}^n {\frac{{{\mathbf{E}_{{\mathfrak{U}_t}}}[ {{{\| {\varepsilon _{i,t}^x} \|}^2}} ]}}{{2{\alpha _{t + 1}}}}}, \label{lemma4-eq1}
\end{flalign}
where ${\mathfrak{U}_t}$ is the $\sigma$-algebra induced by the independent and identically distributed variables ${u_{1,t}}, \cdot  \cdot  \cdot ,{u_{n,t}}$.
\end{lemma}
\begin{proof}
We first analyze the behavior of Algorithm~1 under event-triggering check.

In Algorithm~1, for any $t \in {\mathbb{N}_ + }$, if $\| {{x_{i,t + 1}} - {{\hat x}_{i,t}}} \| \ge {\tau _{t + 1}}$, then $\| {{{\hat x}_{i,t + 1}} - {x_{i,t + 1}}} \| \le {\tau _{t + 1}}$. If $\| {{x_{i,t + 1}} - {{\hat x}_{i,t}}} \| < {\tau _{t + 1}}$, then ${{\hat x}_{i,t + 1}} = {{\hat x}_{i,t}}$ and we still have $\| {{{\hat x}_{i,t + 1}} - {x_{i,t + 1}}} \| \le {\tau _{t + 1}}$. Therefore, we always have $\| {{{\hat x}_{i,t}} - {x_{i,t}}} \| \le {\tau _t}$ for any $t \ge 2$, $i \in [ n ]$. Recall that ${{\hat x}_{i,1}} = {x_{i,1}}$. Thus, $\| {{{\hat x}_{i,t}} - {x_{i,t}}} \| \le {\tau _t}$ for any $t \ge 1$, $i \in [ n ]$.

Next, we give the bound for the average of network-wide loss at one iteration.

From Assumption 1, for $i \in [ n ]$, $t \in {\mathbb{N}_ + }$, $x,y \in \mathbb{X}$, we have
\begin{subequations}
\begin{flalign}
| {{f_{i,t}}( x ) - {f_{i,t}}( y )} | &\le {F_2}\| {x - y} \|, \label{lemma4-proof-eq1a} \\
\| {{g_{i,t}}( x ) - {g_{i,t}}( y )} \| &\le {F_2}\| {x - y} \|. \label{lemma4-proof-eq1b}
\end{flalign}
\end{subequations}

From \eqref{lemma1-eq1b}, \eqref{lemma4-proof-eq1a} and $\| {{{\hat x}_{i,t}} - {x_{i,t}}} \| \le {\tau _t}$, we have
\begin{flalign}
\nonumber
&\;\;\;\;\;\frac{1}{n}\sum\limits_{i = 1}^n {{f_t}( {{x_{i,t}}} )} \\
\nonumber
& = \frac{1}{n}\sum\limits_{i = 1}^n {{f_{i,t}}( {{x_{i,t}}} )}  + \frac{1}{{{n^2}}}\sum\limits_{i = 1}^n {\sum\limits_{j = 1}^n {\big( {{f_{j,t}}( {{x_{i,t}}} ) - {f_{j,t}}( {{x_{j,t}}} )} \big)} } \\
\nonumber
& \le \frac{1}{n}\sum\limits_{i = 1}^n {{f_{i,t}}( {{x_{i,t}}} )}  + \frac{{{F_2}}}{{{n^2}}}\sum\limits_{i = 1}^n {\sum\limits_{j = 1}^n {\| {{x_{i,t}} - {x_{j,t}}} \|} } \\
\nonumber
& \le \frac{1}{n}\sum\limits_{i = 1}^n {{f_{i,t}}( {{x_{i,t}}} )}  + \frac{{2{F_2}}}{n}\sum\limits_{i = 1}^n {\| {{{\hat x}_{i,t}} - {{\bar x}_t}} \|}  + 2{F_2}{\tau _t} \\
&\le \frac{1}{n}\sum\limits_{i = 1}^n {{{\hat f}_{i,t}}( {{x_{i,t}}} )}  + \frac{{2{F_2}}}{n}\sum\limits_{i = 1}^n {\| {{{\hat x}_{i,t}} - {{\bar x}_t}} \|}  + 2{F_2}{\tau _t}. \label{lemma4-proof-eq2}
\end{flalign}

It then follows from the proof of Lemma 10 of \cite{Yi2023} that \eqref{lemma4-eq1} holds.
\end{proof}

\begin{lemma}\label{lem8}
Suppose Assumptions 1 and 2 hold, and ${\gamma _t}{\beta _t} \le 1$, $t \in {\mathbb{N}_ + }$. For all $i \in [ n ]$, let $\{ {{x_{i,t}}}\}$ be the sequences generated by Algorithm~2. Then, for any comparator sequence ${y_{[ T ]}} \in~{\mathcal{X}_T}$,
\begin{flalign}
\nonumber
& \;\;\;\;\;\mathbf{E}[{{\rm{Net}\mbox{-}\rm{Reg}}( {\{ {{x_{i,t}}} \},{y_{[ T ]}}} )}] \\
\nonumber
& \le 2( {p + 1} ){F_2}{\varpi _2} + 2\hat \varpi _1^2\sum\limits_{t = 1}^T {{\gamma _t}}  + 10{{\hat \varpi }_3}\sum\limits_{t = 1}^T {{\alpha _t}}
+ {F_2}( {{{\hat \varpi }_2} + 2} )\sum\limits_{t = 1}^T {{\tau _t}}  + 2R( \mathbb{X} )\sum\limits_{t = 1}^T {\frac{{{\tau _{t + 1}}}}{{{\alpha _{t + 1}}}}  }  \\
\nonumber
&  \;\; + \frac{{2R{{( \mathbb{X} )}^2}}}{{{\alpha _{T + 1}}}}  + \frac{{2R( \mathbb{X} )}}{{{\alpha _T}}}{P_T} + \sum\limits_{t = 1}^T {{F_2}( {R( \mathbb{X} ){\xi _t} + {\delta _t}} )( {\frac{{{{\hat \varpi }_1}}}{{{\beta _t}}} + 1} )}
+ 2R{( \mathbb{X} )^2}\sum\limits_{t = 1}^T {\frac{{{\xi _t} - {\xi _{t + 1}}}}{{{\alpha _{t + 1}}}}} \\
&  \;\; - \frac{1}{{2n}}\sum\limits_{t = 1}^T {\sum\limits_{i = 1}^n {( {\frac{1}{{{\gamma _t}}} - \frac{1}{{{\gamma _{t + 1}}}} + {\beta _{t + 1}}} )\mathbf{E}[ {{{\| {{q_{i,t}}} \|}^2}} ]} }, \label{lemma5-eq1} \\
\nonumber
& \;\;\;\;\;\mathbf{E}[\frac{1}{n}\sum\limits_{i = 1}^n {{\| {\sum\limits_{t = 1}^T {{{[ {{g_t}( {{x_{i,t}}} )} ]}_ + }} } \|^2}}] \\
\nonumber
& \le 4( {p + 1} )n{F_1}{F_2}{\varpi _2}T + 2n{F_1}{F_2}( {{{\hat \varpi }_2} + 2} )T\sum\limits_{t = 1}^T {{\tau _t}}
 + 2\big( {\frac{1}{{{\gamma _1}}} + \sum\limits_{t = 1}^T {( {{\beta _t} + 40{{\hat \varpi }_3}{\alpha _t}} )} } \big)\big( {n{F_1}T } \\
\nonumber
&\;\; + 2( {p + 1} )n{F_2}{\varpi _2} + 2n\hat \varpi _1^2\sum\limits_{t = 1}^T {{\gamma _t}}  + 20n{{\hat \varpi }_3}\sum\limits_{t = 1}^T {{\alpha _t}}  + n{F_2}( {{{\hat \varpi }_2} + 2} )\sum\limits_{t = 1}^T {{\tau _t}} + 2nR( \mathbb{X} )\sum\limits_{t = 1}^T {\frac{{{\tau _{t + 1}}}}{{{\alpha _{t + 1}}}}} \\
\nonumber
&\;\;  + \frac{{2nR{{( \mathbb{X} )}^2}}}{{{\alpha _{T + 1}}}} + 2nR{{( \mathbb{X} )}^2}\sum\limits_{t = 1}^T {\frac{ {\xi _t} - {\xi _{t + 1}} }{{{\alpha _{t + 1}}}}}
 + \sum\limits_{t = 1}^T {n{F_2}\big( {R( \mathbb{X} ){\xi _t} + {\delta _t}} \big)( {\frac{{{{\hat \varpi }_1}}}{{{\beta _t}}} + 1} )} \\
&\;\; - \frac{1}{2}\sum\limits_{t = 1}^T {\sum\limits_{i = 1}^n {( {\frac{1}{{{\gamma _t}}} - \frac{1}{{{\gamma _{t + 1}}}} + {\beta _{t + 1}}} ){{\| {{q_{i,t}} - \hat{\mu} _{ij}^0} \|}^2}} }\big), \label{lemma5-eq2}
\end{flalign}
where ${{\hat \varpi }_2} = 2{\varpi _4} + 2p{\varpi _4} + p$, ${{\hat \varpi }_3} = 2{F_2}{\varpi _3} + {p^2}{\varpi _5}$, and $\hat{\mu} _{ij}^0 = \frac{{\sum\nolimits_{t = 1}^T {{{[ {{g_{i,t}}( {{x_{j,t}}} )} ]}_ + }} }}{{\frac{1}{{{\gamma _1}}} + \sum\nolimits_{t = 1}^T {( {{\beta _t} + 40{{\hat \varpi }_3}{\alpha _t}} )} }}$.
\end{lemma}
\begin{proof}
From \eqref{ass-eq1} and ${{\hat y}_t} = \left( {1 - {\xi _t}} \right){y_t}$, we have
\begin{flalign}
\nonumber
&\;\;\;\;\; {\| {{{\hat y}_{t + 1}} - {{\hat x}_{i,t + 1}}} \|^2} - {\| {{{\hat y}_t} - {x_{i,t + 1}}} \|^2} \\
\nonumber
& = \| {{{\hat y}_{t + 1}} - {{\hat x}_{i,t + 1}} + {{\hat y}_t} - {x_{i,t + 1}}} \|\| {{{\hat y}_{t + 1}} - {{\hat y}_t} + {x_{i,t + 1}} - {{\hat x}_{i,t + 1}}} \| \\
\nonumber
& \le 4R( \mathbb{X} )\| {{{\hat y}_{t + 1}} - {{\hat y}_t} + {x_{i,t + 1}} - {{\hat x}_{i,t + 1}}} \| \\
\nonumber
& \le 4R( \mathbb{X} )\| {{{\hat y}_{t + 1}} - {{\hat y}_t}} \| + 4R( \mathbb{X} )\| {{x_{i,t + 1}} - {{\hat x}_{i,t + 1}}} \| \\
\nonumber
& \le 4R( \mathbb{X} )\| {{{\hat y}_{t + 1}} - {{\hat y}_t}} \| + 4R( \mathbb{X} ){\tau _{t + 1}} \\
\nonumber
& = 4R( \mathbb{X} )\| {( {1 - {\xi _{t + 1}}} ){y_{t + 1}} - ( {1 - {\xi _t}} ){y_t}} \| + 4R( \mathbb{X} ){\tau _{t + 1}} \\
\nonumber
& = 4R( \mathbb{X} )\| {( {1 - {\xi _{t + 1}}} ){y_{t + 1}}} - ( {1 - {\xi _{t + 1}}} ){y_t} + ( {1 - {\xi _{t + 1}}} ){y_t}
  - ( {1 - {\xi _t}} ){y_t} \| + 4R( \mathbb{X} ){\tau _{t + 1}} \\
\nonumber
& = 4R( \mathbb{X} )\| {( {1 - {\xi _{t + 1}}} )( {{y_{t + 1}} - {y_t}} ) + ( {{\xi _t} - {\xi _{t + 1}}} ){y_t}} \|
 + 4R( \mathbb{X} ){\tau _{t + 1}} \\
& \le 4R( \mathbb{X} )\| {{y_{t + 1}} - {y_t}} \| + 4R{( \mathbb{X} )^2}( {{\xi _t} - {\xi _{t + 1}}} ) + 4R( \mathbb{X} ){\tau _{t + 1}}. \label{lemma5-proof-eq1}
\end{flalign}

From \eqref{lemma3-eq2}, \eqref{lemma4-eq1} and \eqref{lemma5-proof-eq1}, we have
\begin{flalign}
\nonumber
& \;\;\;\;\;\frac{1}{n}\sum\limits_{i = 1}^n {\big( {{\Delta _{i,t + 1}}( {{\mu _i}} ) + \mu _i^T{{[ {{g_{i,t}}( {{x_{i,t}}} )} ]}_ + } - \frac{1}{2}{\beta _{t + 1}}{{\| {{\mu _i}} \|}^2}} \big)}
+ \frac{1}{n}\sum\limits_{i = 1}^n {{f_t}( {{x_{i,t}}} )}  - {f_t}( {{y_t}} ) \\
\nonumber
& \le 2\hat \varpi _1^2{\gamma _{t + 1}} + 2{F_2}{\tau _t} + \frac{1}{n}\sum\limits_{i = 1}^n {{{\tilde \Delta }_{i,t + 1}}( {{\mu _i}} )}
+ \frac{1}{{2n{\alpha _{t + 1}}}}\sum\limits_{i = 1}^n {( {{{\| {{\hat{y}_t} - {z_{i,t + 1}}} \|}^2} - {{\| {{\hat{y}_{t + 1}} - {z_{i,t + 2}}} \|}^2}} )} \\
\nonumber
& \;\; + \frac{{2R( \mathbb{X} ){\tau _{t+1}}}}{{{\alpha _{t + 1}}}} + \frac{{2R( \mathbb{X} )}}{{{\alpha _{t + 1}}}}\| {{y_{t + 1}} - {y_t}} \| + \frac{{2R{{( \mathbb{X} )}^2}( {{\xi _t} - {\xi _{t + 1}}} )}}{{{\alpha _{t + 1}}}} \\
& \;\; + \frac{1}{n}\sum\limits_{i = 1}^n {{F_2}} \big( {R( \mathbb{X} ){\xi _t} + {\delta _t}} \big)\big( {\| {{q_{i,t}}} \| + 1} \big), \label{lemma5-proof-eq2}
\end{flalign}
where
\begin{flalign}
\nonumber
{{\tilde \Delta }_{i,t + 1}}( {{\mu _i}} ) &= p{F_2}( {\| {{\mu _i}} \| + 1} )\| {{x_{i,t}} - {x_{i,t + 1}}} \|
+ 2{F_2}\| {{{\hat x}_{i,t}} - {{\bar x}_t}} \| - \frac{1}{{2{\alpha _{t + 1}}}}{\| {\varepsilon _{i,t}^x} \|^2}.
\end{flalign}

From \eqref{ass-eq1} and $\{ {{\alpha _t}} \}$ is non-increasing, we have
\begin{flalign}
\nonumber
& \;\;\;\;\;\sum\limits_{t = 1}^T {\frac{1}{{{\alpha _{t + 1}}}}} ( {{{\| {{\hat{y}_t} - {z_{i,t + 1}}} \|}^2} - {{\| {{\hat{y}_{t + 1}} - {z_{i,t + 2}}} \|}^2}} ) \\
\nonumber
& = \sum\limits_{t = 1}^T \big( {\frac{1}{{{\alpha _t}}}{{\| {{\hat{y}_t} - {z_{i,t + 1}}} \|}^2} - \frac{1}{{{\alpha _{t + 1}}}}{{\| {{\hat{y}_{t + 1}} - {z_{i,t + 2}}} \|}^2}}
+ ( {\frac{1}{{{\alpha _{t + 1}}}} - \frac{1}{{{\alpha _t}}}} ){{\| {{\hat{y}_t} - {z_{i,t + 1}}} \|}^2} \big) \\
& \le \frac{1}{{{\alpha _1}}}{\| {{\hat{y}_1} - {z_{i,2}}} \|^2} - \frac{1}{{{\alpha _{T + 1}}}}{\| {{\hat{y}_{T + 1}} - {z_{i,T + 2}}} \|^2}
+ \sum\limits_{t = 1}^T {( {\frac{1}{{{\alpha _{t + 1}}}} - \frac{1}{{{\alpha _t}}}} )4R{( \mathbb{X} )^2}} \le \frac{{4R{( \mathbb{X} )^2}}}{{{\alpha _{T + 1}}}}. \label{lemma5-proof-eq3}
\end{flalign}

From \eqref{lemma5-proof-eq3} and $\{ {{\alpha _t}} \}$ is non-increasing, setting ${y_{T + 1}} = {y_T}$ and ${\mu _i} = {\mathbf{0}_{{m_i}}}$, summing \eqref{lemma5-proof-eq2} over $t \in [ T ]$ gives
\begin{flalign}
\nonumber
& \;\;\;\;\;{{\rm{Net}\mbox{-}\rm{Reg}}( {\{ {{x_{i,t}}} \},{y_{[ T ]}}} )}  + \frac{1}{n}\sum\limits_{t = 1}^T {\sum\limits_{i = 1}^n {{\Delta _{i,t + 1}}( {{\mathbf{0}_{{m_i}}}} )} }  \\
\nonumber
& \le {2\hat \varpi _1^2}\sum\limits_{t = 1}^T {{\gamma _{t + 1}}} + \sum\limits_{t = 1}^T {2{F_2}{\tau _t}} + \frac{1}{n}\sum\limits_{t = 1}^T {\sum\limits_{i = 1}^n {{{\tilde \Delta }_{i,t + 1}}( {{\mathbf{0}_{{m_i}}}} )} }
+ \sum\limits_{t = 1}^T {\frac{{2R( \mathbb{X} ){\tau _{t+1}}}}{{{\alpha _{t + 1}}}}}  + \frac{{2R{{( \mathbb{X} )}^2}}}{{{\alpha _{T+1}}}} + \frac{{2R( \mathbb{X} )}}{{{\alpha _T}}}{P_T}  \\
& \;\; + 2R{{( \mathbb{X} )}^2}\sum\limits_{t = 1}^T {\frac{ {\xi _t} - {\xi _{t + 1}} }{{{\alpha _{t + 1}}}}}
+ \sum\limits_{t = 1}^T {{F_2}\big( {R( \mathbb{X} ){\xi _t} + {\delta _t}} \big)( {\frac{{{{\hat \varpi }_1}}}{{{\beta _t}}} + 1} )}. \label{lemma5-proof-eq4}
\end{flalign}

We then establish a lower bound for network regret.

For any $T \in {\mathbb{N}_ + }$, we have
\begin{flalign}
 \sum\limits_{t = 1}^T {{\Delta _{i,t + 1}}( {{\mu _i}} )}
& = \frac{1}{2}\sum\limits_{t = 1}^T {( {\frac{1}{{{\gamma _t}}} - \frac{1}{{{\gamma _{t + 1}}}} + {\beta _{t + 1}}} ){{\| {{q_{i,t}} - {\mu _i}} \|}^2}}
+ \frac{{{{\| {{q_{i,T + 1}} - {\mu _i}} \|}^2}}}{{2{\gamma _{T + 1}}}} - \frac{{{{\| {{\mu _i}} \|}^2}}}{{2{\gamma _1}}}. \label{lemma5-proof-eq5}
\end{flalign}

Substituting ${\mu _i} = {\mathbf{0}_{{m_i}}}$ into \eqref{lemma5-proof-eq5} yields
\begin{flalign}
\sum\limits_{t = 1}^T {{\Delta _{i,t + 1}}( {{\mathbf{0}_{{m_i}}}} )}
\ge \frac{1}{2}\sum\limits_{t = 1}^T {( {\frac{1}{{{\gamma _t}}} - \frac{1}{{{\gamma _{t + 1}}}} + {\beta _{t + 1}}} ){{\| {{q_{i,t}}} \|}^2}}. \label{lemma5-proof-eq6}
\end{flalign}

We have
\begin{flalign}
\sum\limits_{t = 1}^T {\sum\limits_{s = 1}^{t - 2} {{\lambda ^{t - s - 2}}\sum\limits_{j = 1}^n {\| {\hat{\varepsilon} _{j,s}^x} \|} } }  & = \sum\limits_{t = 1}^{T - 2} {\sum\limits_{j = 1}^n {\| {\hat{\varepsilon} _{j,t}^x} \|\sum\limits_{s = 0}^{T - t - 2} {{\lambda ^s}} } }  \le \frac{1}{{1 - \lambda }}\sum\limits_{t = 1}^{T - 2} {\sum\limits_{j = 1}^n {\| {\hat{\varepsilon} _{j,t}^x} \|} }, \label{lemma5-proof-eq7}
\end{flalign}
and
\begin{flalign}
\| {\hat \varepsilon _{i,t - 1}^x} \| &= \| {{{\hat x}_{i,t}} - {x_{i,t}} + {x_{i,t}} - {z_{i,t}}} \| \le \| {\varepsilon _{i,t - 1}^x} \| + {\tau _t}, \label{lemma5-proof-eq8}
\end{flalign}
where ${\varepsilon _{i,t - 1}^x} = {x_{i,t}} - {z_{i,t}}$.

From \eqref{lemma2-eq1}, \eqref{lemma5-proof-eq7} and \eqref{lemma5-proof-eq8}, we have
\begin{flalign}
\nonumber
 \sum\limits_{t = 1}^T {\| {{{\hat x}_{i,t}} - {{\bar x}_t}} \|}
&\le {\varpi _2} + \frac{1}{n}\sum\limits_{t = 1}^T {\sum\limits_{j = 1}^n {\| {\varepsilon _{j,t - 1}^x} \|} }  + \sum\limits_{t = 1}^T {\| {\varepsilon _{i,t - 1}^x} \|} + 2\sum\limits_{t = 1}^T {{\tau _t}}\\
& \;\;+ \frac{\tau }{{1 - \lambda }}\sum\limits_{t = 1}^{T - 2} {\sum\limits_{j = 1}^n {\| {\varepsilon _{j,t}^x} \|} }    + \frac{{n\tau }}{{1 - \lambda }}\sum\limits_{t = 1}^{T - 2} {{\tau _{t + 1}}}. \label{lemma5-proof-eq9}
\end{flalign}

For any ${\mu _i} \in {\mathbb{R}^{{m_i}}}$ and $a > 0$, it follows from \eqref{lemma5-proof-eq9} that
\begin{flalign}
\nonumber
& \;\;\;\;\;\sum\limits_{t = 1}^T {\sum\limits_{i = 1}^n {\| {{\mu _i}} \|\| {{{\hat x}_{i,t}} - {{\bar x}_t}} \|} } \\
\nonumber
& \le {\varpi _2}\sum\limits_{i = 1}^n {\| {{\mu _i}} \|}
+ \frac{1}{n}\sum\limits_{t = 2}^T {\sum\limits_{i = 1}^n {\sum\limits_{j = 1}^n {( {\frac{1}{{4ap{F_2}{\alpha _t}}}{{\| {\varepsilon _{i,t - 1}^x} \|}^2} + ap{F_2}{\alpha _t}{{\| {{\mu _j}} \|}^2}} )} } } \\
\nonumber
&\;\; + \sum\limits_{t = 2}^T {\sum\limits_{i = 1}^n {( {\frac{1}{{4ap{F_2}{\alpha _t}}}{{\| {\varepsilon _{i,t - 1}^x} \|}^2} + ap{F_2}{\alpha _t}{{\| {{\mu _i}} \|}^2}} )} } \\
\nonumber
&\;\; + \sum\limits_{t = 2}^T {\sum\limits_{i = 1}^n {\sum\limits_{j = 1}^n {( {\frac{1}{{2anp{F_2}{\alpha _t}}}{{\| {\varepsilon _{i,t - 1}^x} \|}^2} + \frac{{anp{F_2}{\tau ^2}{\alpha _t}}}{{2{{( {1 - \lambda } )}^2}}}{{\| {{\mu _j}} \|}^2}} )} } } \\
\nonumber
&\;\; + \sum\limits_{t = 1}^T {\sum\limits_{i = 1}^n {2{\tau _t}\| {{\mu _i}} \|} }  + \frac{{n\tau }}{{1 - \lambda }}\sum\limits_{t = 2}^T {\sum\limits_{i = 1}^n {{\tau _t}\| {{\mu _i}} \|} } \\
& \le {\varpi _2}\sum\limits_{i = 1}^n {\| {{\mu _i}} \|} + \sum\limits_{t = 2}^T {\sum\limits_{i = 1}^n {\frac{1}{{ap{F_2}{\alpha _t}}}{{\| {\varepsilon _{i,t - 1}^x} \|}^2}} }
+ \sum\limits_{t = 2}^T {\sum\limits_{i = 1}^n {ap{\varpi _3}{\alpha _t}{{\| {{\mu _i}} \|}^2}} } + {\varpi _4}\sum\limits_{t = 1}^T {\sum\limits_{i = 1}^n {{\tau _t}\| {{\mu _i}} \|} }. \label{lemma5-proof-eq10}
\end{flalign}

For any ${\mu _i} \in {\mathbb{R}^{{m_i}}}$ and $a > 0$, we have
\begin{flalign}
\nonumber
& \;\;\;\;\;\| {{\mu _i}} \|\| {{x_{i,t}} - {x_{i,t + 1}}} \| \\
\nonumber
& \le \| {{\mu _i}} \|\| {{x_{i,t}} - {z_{i,t + 1}}} \| + \| {{\mu _i}} \|\| {{z_{i,t + 1}} - {x_{i,t + 1}}} \| \\
& \le \| {{\mu _i}} \|\| {{x_{i,t}} - {z_{i,t + 1}}} \|
+ \frac{1}{{ap{F_2}{\alpha _{t + 1}}}}{\| {\varepsilon _{i,t}^x} \|^2} + \frac{{ap{F_2}{\alpha _{t + 1}}}}{4}{\| {{\mu _i}} \|^2}. \label{lemma5-proof-eq11}
\end{flalign}

From \eqref{Algorithm1-eq1} and ${\sum\nolimits_{i = 1}^n {[ {{W_t}} ]} _{ij}} = {\sum\nolimits_{j = 1}^n {[ {{W_t}} ]} _{ij}} = 1$, we have
\begin{flalign}
\nonumber
& \;\;\;\;\;\sum\limits_{i = 1}^n {\| {{x_{i,t}} - {z_{i,t + 1}}} \|} \\
\nonumber
& \le \sum\limits_{i = 1}^n {( {\| {{x_{i,t}} - {{\bar x}_t}} \| + \| {{{\bar x}_t} - {z_{i,t + 1}}} \|} )} \\
\nonumber
& \le \sum\limits_{i = 1}^n {( {\| {{x_{i,t}} - {{\hat x}_{i,t}} + {{\hat x}_{i,t}} - {{\bar x}_t}} \| + \| {{{\bar x}_t} - \sum\limits_{j = 1}^n {{{[ {{W_t}} ]}_{ij}}{{\hat x}_{j,t}}} } \|} )} \\
\nonumber
& \le \sum\limits_{i = 1}^n {( {\| {{{\hat x}_{i,t}} - {x_{i,t}}} \| + \| {{{\hat x}_{i,t}} - {{\bar x}_t}} \|} )}
+ \sum\limits_{i = 1}^n {\sum\limits_{j = 1}^n {{{[ {{W_t}} ]}_{ij}}\| {{{\bar x}_t} - {{\hat x}_{j,t}} } \|} } \\
& \le 2\sum\limits_{i = 1}^n {\| {{{\hat x}_{i,t}} - {{\bar x}_t}} \|}  + \sum\limits_{i = 1}^n {\tau _t}. \label{lemma5-proof-eq12}
\end{flalign}
From \eqref{lemma5-proof-eq10}--\eqref{lemma5-proof-eq12}, for any ${\mu _i} \in {\mathbb{R}^{{m_i}}}$ and $a > 0$, we have
\begin{flalign}
\nonumber
& \;\;\;\;\;\sum\limits_{t = 1}^T {\sum\limits_{i = 1}^n {p{F_2}\| {{\mu _i}} \|\| {{x_{i,t}} - {x_{i,t + 1}}} \|} } \\
\nonumber
& \le \sum\limits_{t = 1}^T {\sum\limits_{i = 1}^n {2p{F_2}\| {{\mu _i}} \|\| {{\hat{x}_{i,t}} - {{\bar x}_t}} \|} } + p{F_2}\sum\limits_{t = 1}^T {\sum\limits_{i = 1}^n {{\tau _t}\| {{\mu _i}} \|} } + \sum\limits_{t = 1}^T {\sum\limits_{i = 1}^n {\frac{1}{{a{\alpha _{t + 1}}}}{{\| {\varepsilon _{i,t}^x} \|}^2}} } \\
\nonumber
&\;\;   + \sum\limits_{t = 1}^T {\sum\limits_{i = 1}^n {\frac{{ap^2F_2^2{\alpha _{t + 1}}}}{4}{{\| {{\mu _i}} \|}^2}} } \\
\nonumber
& \le 2p{F_2}{\varpi _2}\sum\limits_{i = 1}^n {\| {{\mu _i}} \|}  + ( {2{\varpi _4} + 1} )p{F_2}\sum\limits_{t = 1}^T {\sum\limits_{i = 1}^n {{\tau _t}\| {{\mu _i}} \|} } + \sum\limits_{t = 1}^T {\sum\limits_{i = 1}^n {ap^2{\varpi _5}{\alpha _{t + 1}}{{\| {{\mu _i}} \|}^2}} }\\
& \;\; + \sum\limits_{t = 1}^T {\sum\limits_{i = 1}^n {\frac{3}{{a{\alpha _{t + 1}}}}{{\| {\varepsilon _{i,t}^x} \|}^2}} }. \label{lemma5-proof-eq13}
\end{flalign}

Choosing $\| {{\mu _i}} \| = 1$ in \eqref{lemma5-proof-eq10} and \eqref{lemma5-proof-eq13} yields
\begin{flalign}
\nonumber
& \;\;\;\;\;\sum\limits_{t = 1}^T {\sum\limits_{i = 1}^n {2{F_2}\| {{{\hat x}_{i,t}} - {{\bar x}_t}} \|} } \\
& \le 2n{F_2}{\varpi _2} + 2n{F_2}{\varpi _4}\sum\limits_{t = 1}^T {{\tau _t}}
+ \sum\limits_{t = 2}^T {\sum\limits_{i = 1}^n {( {2a{F_2}{\varpi _3}{\alpha _t} + \frac{2}{{a{\alpha _t}}}{{\| {\varepsilon _{i,t - 1}^x} \|}^2}} )} }, \label{lemma5-proof-eq14}
\end{flalign}
and
\begin{flalign}
\nonumber
& \;\;\;\;\;\sum\limits_{t = 1}^T {\sum\limits_{i = 1}^n {p{F_2}\| {{x_{i,t}} - {x_{i,t + 1}}} \|} } \\
& \le 2np{F_2}{\varpi _2} + ( {2{\varpi _4} + 1} )np{F_2}\sum\limits_{t = 1}^T {{\tau _t}} + an{p^2}{\varpi _5}\sum\limits_{t = 1}^T {\alpha _{t + 1}}
+ \sum\limits_{t = 1}^T {\sum\limits_{i = 1}^n {\frac{3}{{a{\alpha _{t + 1}}}}{{\left\| {\varepsilon _{i,t}^x} \right\|}^2}} }. \label{lemma5-proof-eq15}
\end{flalign}

From \eqref{lemma5-proof-eq13}--\eqref{lemma5-proof-eq15}, and choosing $a = 10$ yields
\begin{flalign}
\nonumber
& \;\;\;\;\;\sum\limits_{t = 1}^T {\sum\limits_{i = 1}^n {{{\tilde \Delta }_{i,t + 1}}( {{\mathbf{0}_{{m_i}}}} )} } \\
\nonumber
& \le 2np{F_2}{\varpi _2} + ( {2{\varpi _4} + 1} )np{F_2}\sum\limits_{t = 1}^T {{\tau _t}}  + 10n{p^2}{\varpi _5}\sum\limits_{t = 1}^T {{\alpha _{t + 1}}} + \sum\limits_{t = 1}^T {\sum\limits_{i = 1}^n {\frac{3}{{10{\alpha _{t + 1}}}}{{\| {\varepsilon _{i,t}^x} \|}^2}} }  + 2n{F_2}{\varpi _2}\\
\nonumber
& \;\;  + 2n{F_2}{\varpi _4}\sum\limits_{t = 1}^T {{\tau _t}}
 + 20n{F_2}{\varpi _3}\sum\limits_{t = 2}^T {{\alpha _t}}  + \sum\limits_{t = 2}^T {\sum\limits_{i = 1}^n {\frac{2}{{10{\alpha _t}}}{{\| {\varepsilon _{i,t - 1}^x} \|}^2}} }  - \frac{1}{{2{\alpha _{t + 1}}}}{\| {\varepsilon _{i,t}^x} \|^2} \\
&  \le 2( {p + 1} )n{F_2}{\varpi _2} + n{F_2}{\hat \varpi _2}\sum\limits_{t = 1}^T {{\tau _t}}  + 10n{\hat \varpi _3}\sum\limits_{t = 1}^T {{\alpha _t}} . \label{lemma5-proof-eq16}
\end{flalign}

Combining \eqref{lemma5-proof-eq4}, \eqref{lemma5-proof-eq6} and \eqref{lemma5-proof-eq16} yields \eqref{lemma5-eq1}.

($\mathbf{ii}$)
We first provide a loose bound for network cumulative constraint violation.

We have
\begin{flalign}
\nonumber
\mu _i^T{[ {{g_{i,t}}( {{x_{i,t}}} )} ]_ + }
&= \mu _i^T{[ {{g_{i,t}}( {{x_{j,t}}} )} ]_ + } + \mu _i^T{[ {{g_{i,t}}( {{x_{i,t}}} )} ]_ + } - \mu _i^T{[ {{g_{i,t}}( {{x_{j,t}}} )} ]_ + } \\
\nonumber
& \ge \mu _i^T{[ {{g_{i,t}}( {{x_{j,t}}} )} ]_ + } - \| {{\mu _i}} \|\| {{{[ {{g_{i,t}}( {{x_{i,t}}} )} ]}_ + } - {{[ {{g_{i,t}}( {{x_{j,t}}} )} ]}_ + }} \| \\
\nonumber
& \ge \mu _i^T{[ {{g_{i,t}}( {{x_{j,t}}} )} ]_ + } - \| {{\mu _i}} \|\| {{g_{i,t}}( {{x_{i,t}}} ) - {g_{i,t}}( {{x_{j,t}}} )} \| \\
\nonumber
& \ge \mu _i^T{[ {{g_{i,t}}( {{x_{j,t}}} )} ]_ + } - {F_2}\| {{\mu _i}} \|\| {{x_{i,t}} - {x_{j,t}}} \| \\
\nonumber
& \ge \mu _i^T{[ {{g_{i,t}}( {{x_{j,t}}} )} ]_ + }
- {F_2}\| {{\mu _i}} \|\| {{{\hat x}_{i,t}} - {x_{i,t}}} \| - {F_2}\| {{\mu _i}} \|\| {{{\hat x}_{i,t}} - {{\bar x}_t}} \| \\
& \;\; - {F_2}\| {{\mu _i}} \|\| {{{\hat x}_{j,t}} - {x_{j,t}}} \| - {F_2}\| {{\mu _i}} \|\| {{{\hat x}_{j,t}} - {{\bar x}_t}} \|, \label{lemma5-proof-eq17}
\end{flalign}
where the second inequality holds since projection operator is non-expansive, and the third inequality holds due to \eqref{lemma4-proof-eq1b}.

Combining \eqref{lemma5-proof-eq2} and \eqref{lemma5-proof-eq17}, setting ${y_t} = y$, and summing over $j \in [ n ]$ yields
\begin{flalign}
\nonumber
& \;\;\;\;\;\sum\limits_{i = 1}^n {\big( {{\Delta _{i,t + 1}}( {{\mu _i}} ) + \frac{1}{n}\sum\limits_{j = 1}^n {\mu _i^T{{[ {{g_{i,t}}( {{x_{j,t}}} )} ]}_ + }}  - \frac{1}{2}{\beta _{t + 1}}{{\| {{\mu _i}} \|}^2}} \big)}
+ \sum\limits_{i = 1}^n {{f_t}( {{x_{i,t}}} )}  - n{f_t}( y ) \\
\nonumber
& \le 2n\hat \varpi _1^2{\gamma _{t + 1}} + \sum\limits_{i = 1}^n {{{\hat \Delta }_{i,t + 1}}\left( {{\mu _i}} \right)}  + \frac{1}{n}{\check{\Delta} _t} + 2n{F_2}{\tau _t}
+ \frac{{2nR( \mathbb{X} ){\tau _{t + 1}}}}{{{\alpha _{t + 1}}}} + \frac{{2nR{{( \mathbb{X} )}^2}\left( {{\xi _t} - {\xi _{t + 1}}} \right)}}{{{\alpha _{t + 1}}}} \\
& \;\; + \frac{1}{{2{\alpha _{t + 1}}}}\sum\limits_{i = 1}^n {( {{{\| {{{\hat y}_t} - {z_{i,t + 1}}} \|}^2} - {{\| {{{\hat y}_{t + 1}} - {z_{i,t + 2}}} \|}^2}} )}  + \sum\limits_{i = 1}^n {{F_2}( {R( \mathbb{X} ){\xi _t} + {\delta _t}} )( {\| {{q_{i,t}}} \| + 1} )}, \label{lemma5-proof-eq18}
\end{flalign}
where
\begin{flalign}
\nonumber
{{\hat \Delta }_{i,t + 1}}( {{\mu _i}} ) &= {{\tilde \Delta }_{i,t + 1}}( {{\mu _i}} ) + {F_2}\| {{\mu _i}} \|\| {{{\hat x}_{i,t}} - {{\bar x}_t}} \| + {F_2}{\tau _t}\| {{\mu _i}} \|, \\
\nonumber
{\check{\Delta} _t} &= \sum\limits_{i = 1}^n {\sum\limits_{j = 1}^n {{F_2}\| {{\mu _i}} \|\| {{{\hat x}_{j,t}} - {{\bar x}_t}} \|} }  + \sum\limits_{i = 1}^n {n{F_2}{\tau _t}\| {{\mu _i}} \|}.
\end{flalign}

Similar to the way to \eqref{lemma5-proof-eq10}, we have
\begin{flalign}
\nonumber
& \;\;\;\;\;\sum\limits_{t = 1}^T {\sum\limits_{i = 1}^n {F_2\| {{\mu _i}} \|\| {{{\hat x}_{i,t}} - {{\bar x}_t}} \|} } \\
& \le {F_2}{\varpi _2}\sum\limits_{i = 1}^n {\| {{\mu _i}} \|} + \sum\limits_{t = 2}^T {\sum\limits_{i = 1}^n {\frac{1}{{a{\alpha _t}}}{{\| {\varepsilon _{i,t - 1}^x} \|}^2}} }
+ \sum\limits_{t = 2}^T {\sum\limits_{i = 1}^n {a{F_2}{\varpi _3}{\alpha _t}{{\| {{\mu _i}} \|}^2}} } + {F_2}{\varpi _4}\sum\limits_{t = 1}^T {\sum\limits_{i = 1}^n {{\tau _t}\| {{\mu _i}} \|} }. \label{lemma5-proof-eq19}
\end{flalign}

Combining \eqref{lemma5-proof-eq10}, \eqref{lemma5-proof-eq13}--\eqref{lemma5-proof-eq15}, \eqref{lemma5-proof-eq19}, and choosing $a = 20$ yields
\begin{flalign}
\nonumber
& \;\;\;\;\;\sum\limits_{t = 1}^T {\sum\limits_{i = 1}^n {{{\hat \Delta }_{i,t + 1}}( {{\mu _i}} )} } \\
\nonumber
& \le 2( {p + 1} )n{F_2}{\varpi _2} + 20n{{\hat \varpi }_3}\sum\limits_{t = 1}^T {{\alpha _t}} + n{F_2}{{\hat \varpi }_2}\sum\limits_{t = 1}^T {{\tau _t}}
+ ( {2p + 1} ){F_2}{\varpi _2}\sum\limits_{i = 1}^n {\| {{\mu _i}} \|}  \\
\nonumber
& \;\; + 20( {{p^2}{\varpi _5} + {F_2}{\varpi _3}} )\sum\limits_{t = 1}^T {\sum\limits_{i = 1}^n {{\alpha _t}{{\| {{\mu _i}} \|}^2}} } + {F_2}( {2p{\varpi _4} + p + {\varpi _4} + 1} )\sum\limits_{t = 1}^T {\sum\limits_{i = 1}^n {{\tau _t}\| {{\mu _i}} \|} } \\
& \;\; - \sum\limits_{t = 1}^T {\sum\limits_{i = 1}^n {\frac{1}{{20{\alpha _{t + 1}}}}} } {\| {\varepsilon _{i,t}^x} \|^2}. \label{lemma5-proof-eq20}
\end{flalign}

From \eqref{lemma5-proof-eq9}, we have
\begin{flalign}
\nonumber
& \;\;\;\;\;\sum\limits_{t = 1}^T {\sum\limits_{i = 1}^n {\sum\limits_{j = 1}^n {\| {{\mu _j}} \|\| {{{\hat x}_{i,t}} - {{\bar x}_t}} \|} } } \\
\nonumber
& \le n{\varpi _2}\sum\limits_{j = 1}^n {\| {{\mu _j}} \|}  + 2\sum\limits_{t = 2}^T {\sum\limits_{i = 1}^n {\sum\limits_{j = 1}^n {\| {\varepsilon _{i,t - 1}^x} \|\| {{\mu _j}} \|} } } + \frac{{n\tau }}{{1 - \lambda }}\sum\limits_{t = 2}^T {\sum\limits_{i = 1}^n {\sum\limits_{j = 1}^n {\| {\varepsilon _{i,t - 1}^x} \|\| {{\mu _j}} \|} } }\\
\nonumber
&  \;\;  + n{\varpi _4}\sum\limits_{t = 1}^T {\sum\limits_{j = 1}^n {{\tau _t}\| {{\mu _j}} \|} } \\
\nonumber
& \le n{\varpi _2}\sum\limits_{j = 1}^n {\| {{\mu _j}} \|} + \sum\limits_{t = 2}^T {\sum\limits_{i = 1}^n {\sum\limits_{j = 1}^n {( {\frac{1}{{2a{F_2}{\alpha _t}}}{{\| {\varepsilon _{i,t - 1}^x} \|}^2} + 2a{F_2}{\alpha _t}{{\| {{\mu _j}} \|}^2}} )} } } \\
\nonumber
&  \;\;+ \sum\limits_{t = 2}^T {\sum\limits_{i = 1}^n {\sum\limits_{j = 1}^n {( {\frac{1}{{2a{F_2}{\alpha _t}}}{{\| {\varepsilon _{i,t - 1}^x} \|}^2} + \frac{{a{n^2}{F_2}{\tau ^2}{\alpha _t}}}{{2{{( {1 - \lambda } )}^2}}}{{\| {{\mu _j}} \|}^2}} )} } } + n{\varpi _4}\sum\limits_{t = 1}^T {\sum\limits_{j = 1}^n {{\tau _t}\| {{\mu _j}} \|} } \\
\nonumber
& = n{\varpi _2}\sum\limits_{i = 1}^n {\| {{\mu _i}} \|}  + \sum\limits_{t = 2}^T {\sum\limits_{i = 1}^n {an{\varpi _3}{\alpha _t}{{\| {{\mu _i}} \|}^2}} } + \sum\limits_{t = 2}^T {\sum\limits_{i = 1}^n {\frac{n}{{a{F_2}{\alpha _t}}}} } {\| {\varepsilon _{i,t - 1}^x} \|^2}\\
& \;\; + n{\varpi _4}\sum\limits_{t = 1}^T {\sum\limits_{i = 1}^n {{\tau _t}\| {{\mu _i}} \|} }. \label{lemma5-proof-eq21}
\end{flalign}

Choosing $a = 20$ in \eqref{lemma5-proof-eq21} yields
\begin{flalign}
\nonumber
&\;\;\;\;\;\frac{1}{n}\sum\limits_{t = 1}^T {{\check{\Delta} _t}} \\
&\le {F_2}{\varpi _2}\sum\limits_{i = 1}^n {\| {{\mu _i}} \|}  + \sum\limits_{t = 1}^T {\sum\limits_{i = 1}^n {20{F_2}{\varpi _3}{\alpha _t}{{\| {{\mu _i}} \|}^2}} }
+ \sum\limits_{t = 2}^T {\sum\limits_{i = 1}^n {\frac{1}{{20{\alpha _t}}}{{\| {\varepsilon _{i,t - 1}^x} \|}^2}} }  + {F_2}{\varpi _8}\sum\limits_{t = 1}^T {\sum\limits_{i = 1}^n {{\tau _t}\| {{\mu _i}} \|} }. \label{lemma5-proof-eq22}
\end{flalign}

Let ${h_{ij}}:\mathbb{R}_ + ^{{m_i}} \to \mathbb{R}$ be a function defined as
\begin{flalign}
{h_{ij}}( {{\mu _i}} ) &= \mu _i^T\sum\limits_{t = 1}^T {{{[ {{g_{i,t}}( {{x_{j,t}}} )} ]}_ + }}
- \frac{1}{2}{\| {{\mu _i}} \|^2}\big( {\frac{1}{{{\gamma _1}}} + \sum\limits_{t = 1}^T {( {{\beta _t} + 40{{\hat \varpi }_3}{\alpha _t}} )} } \big). \label{lemma5-proof-eq23}
\end{flalign}

From \eqref{lemma5-proof-eq3}, \eqref{lemma5-proof-eq5}, \eqref{lemma5-proof-eq20}, \eqref{lemma5-proof-eq22}, and \eqref{lemma5-proof-eq23}, summing \eqref{lemma5-proof-eq18} over $t \in [ T ]$ gives
\begin{flalign}
\nonumber
& \;\;\;\;\;\frac{1}{2}\sum\limits_{t = 1}^T {\sum\limits_{i = 1}^n {( {\frac{1}{{{\gamma _t}}} - \frac{1}{{{\gamma _{t + 1}}}} + {\beta _{t + 1}}} ){{\| {{q_{i,t}} - {\mu _i}} \|}^2}} }
+ \frac{1}{n}\sum\limits_{i = 1}^n {\sum\limits_{j = 1}^n {{h_{ij}}( {{\mu _i}} )} } + n{{\rm{Net}\mbox{-}\rm{Reg}}( {\{ {{x_{i,t}}} \},{y_{[T]}}} )} \\
\nonumber
&  \le 2( {p + 1} )n{F_2}{\varpi _2} + 2n\hat \varpi _1^2\sum\limits_{t = 1}^T {{\gamma _t}}  + 20n{{\hat \varpi }_3}\sum\limits_{t = 1}^T {{\alpha _t}} + n{F_2}( {{{\hat \varpi }_2} + 2} )\sum\limits_{t = 1}^T {{\tau _t}}  \\
\nonumber
&  \;\; + 2( {p + 1} ){F_2}{\varpi _2}\sum\limits_{i = 1}^n {\| {{\mu _i}} \|}  + {F_2}( {{{\hat \varpi }_2} + 2} )\sum\limits_{t = 1}^T {\sum\limits_{i = 1}^n {{\tau _t}\| {{\mu _i}} \|} }  + 2nR( \mathbb{X} )\sum\limits_{t = 1}^T {\frac{{{\tau _{t + 1}}}}{{{\alpha _{t + 1}}}}} + \frac{{2nR{{( \mathbb{X} )}^2}}}{{{\alpha _{T + 1}}}}\\
&  \;\; + \sum\limits_{t = 1}^T {\frac{{2nR{{( \mathbb{X} )}^2}( {{\xi _t} - {\xi _{t + 1}}} )}}{{{\alpha _{t + 1}}}}} + \sum\limits_{t = 1}^T {n{F_2}( {R( \mathbb{X} ){\xi _t} + {\delta _t}} )( {\frac{{{{\hat \varpi }_1}}}{{{\beta _t}}} + 1} )}. \label{lemma5-proof-eq24}
\end{flalign}

Substituting ${\mu _i} = \mu _{ij}^0$ into \eqref{lemma5-proof-eq23} yields
\begin{flalign}
{h_{ij}}\left( {\mu _{ij}^0} \right) = \frac{{{{\| {\sum\nolimits_{t = 1}^T {{{[ {{g_{i,t}}( {{x_{j,t}}} )} ]}_ + }} } \|}^2}}}{{2\big( {\frac{1}{{{\gamma _1}}} + \sum\nolimits_{t = 1}^T {( {{\beta _t} + 40{{\hat \varpi }_3}{\alpha _t}} )} } \big)}}. \label{lemma5-proof-eq25}
\end{flalign}

From ${g_t}( x ) = {\rm{col}}\big( {{g_{1,t}}( x ), \cdot  \cdot  \cdot ,{g_{n,t}}( x )} \big)$, we have
\begin{flalign}
\sum\limits_{i = 1}^n {\sum\limits_{j = 1}^n {{{\| \sum\limits_{t = 1}^T {{{[ {{g_{i,t}}( {{x_{j,t}}} )} ]}_ + }} } \|^2}} }  = \sum\limits_{j = 1}^n {{\| {\sum\limits_{t = 1}^T {{{[ {{g_t}( {{x_{j,t}}} )} ]}_ + }} } \|^2}}. \label{lemma5-proof-eq26}
\end{flalign}

From \eqref{ass-eq2a}, we have
\begin{flalign}
-{{\rm{Net}\mbox{-}\rm{Reg}}( {\{ {{x_{i,t}}} \},{y_{[T]}}} )} \le {F_1}T. \label{lemma5-proof-eq27}
\end{flalign}

From \eqref{ass-eq2b}, we have
\begin{flalign}
\| {\mu _{ij}^0} \| \le \frac{{{F_1}T}}{{\frac{1}{{{\gamma _1}}} + \sum\nolimits_{t = 1}^T {( {{\beta _t} + 40{{\hat \varpi }_3}{\alpha _t}} )} }}. \label{lemma5-proof-eq28}
\end{flalign}

Substituting ${\mu _i} = \mu _{ij}^0$ into \eqref{lemma5-proof-eq24}, combining \eqref{lemma5-proof-eq25}--\eqref{lemma5-proof-eq28} yields \eqref{lemma5-eq2}.

\hspace{-3mm}\emph{B. Proof of Theorem 1}

Based on Lemma~5, we are now ready to prove Theorem~1.

($\mathbf{i}$)
For any constant $a \in \left[ {0,1} \right)$ and $T \in {\mathbb{N}_ + }$, it holds that
\begin{flalign}
\sum\limits_{t = 1}^T {\frac{1}{{{t^a}}}}  \le 1 + \int\limits_1^T {\frac{1}{{{t^a}}}} dt = \frac{{{T^{1 - a}} - a}}{{1 - a}} \le \frac{{{T^{1 - a}}}}{{1 - a}}. \label{theorem1-proof-eq1}
\end{flalign}

Form \eqref{theorem1-proof-eq1}, we have
\begin{flalign}
\sum\limits_{t = 1}^T {\sqrt {\frac{{{\Psi _t}}}{t}} }  \le \sqrt {{\Psi _T}} \sum\limits_{t = 1}^T {\frac{1}{{\sqrt t }}}  \le 2\sqrt {T{\Psi _T}}. \label{theorem1-proof-eq2}
\end{flalign}

From Cauchy--Schwarz inequality, we have
\begin{flalign}
\sum\limits_{t = 1}^T {\frac{{{\tau _{t + 1}}}}{{\sqrt {\frac{{{\Psi _{t + 1}}}}{{t + 1}}} }}}  \le \sum\limits_{t = 1}^T {\sqrt {{\tau _{t + 1}}} }  \le \sum\limits_{t = 1}^T {\sqrt {{\tau _t}} }  \le \sqrt {T{\Psi _T}}. \label{theorem1-proof-eq3}
\end{flalign}

From \eqref{theorem1-eq1}, we have
\begin{flalign}
\frac{t}{{{t^\kappa }}} - \frac{{t + 1}}{{{{\left( {t + 1} \right)}^\kappa }}} + \frac{1}{{{(t+1)^\kappa }}} = \frac{{t}}{{{t^\kappa }}} - \frac{{t}}{{{{( {t + 1} )}^\kappa }}} > 0. \label{theorem1-proof-eq4}
\end{flalign}

For any $T \in {\mathbb{N}_ + }$, there exists a constant $H > 0$ such that
\begin{flalign}
\sum\limits_{t = 1}^T {( {\frac{1}{{t + 1}} - \frac{1}{{t + 2}}} )\sqrt {\frac{t+1}{{{\Psi _{t+1}}}}} }  \le H\sqrt {\frac{1}{{{\Psi _2}}}}. \label{theorem1-proof-eq5}
\end{flalign}
Combining \eqref{theorem1-eq1}, \eqref{lemma5-eq1}, \eqref{theorem1-proof-eq1}--\eqref{theorem1-proof-eq5} yields
\begin{flalign}
\nonumber
& \;\;\;\;\;\mathbf{E}[{{\rm{Net}\mbox{-}\rm{Reg}}( {\{ {{x_{i,t}}} \},{y_{[ T ]}}} )}] \\
\nonumber
& \le 2( {p + 1} ){F_2}{\varpi _2} + \frac{{2\hat \varpi _1^2}}{\kappa }{T^\kappa } + 20{{\hat \varpi }_3}\sqrt {T{\Psi _T}} + {F_2}( {{{\hat \varpi }_2} + 2} ){\Psi _T} + 2R( \mathbb{X} )\sqrt {T{\Psi _T}} + 2\sqrt 2 R{( \mathbb{X} )^2}\sqrt {\frac{T}{{{\Psi _T}}}}\\
&\;\; + {F_2}\big( {R( \mathbb{X} ) + r( \mathbb{X} )} \big)\big( {\frac{{{{\hat \varpi }_1}}}{\kappa }{T^\kappa } + \log ( T )} \big)  + 2HR{( \mathbb{X} )^2}\sqrt {\frac{1}{{{\Psi _2}}}}  + 2R( \mathbb{X} )\sqrt {\frac{T}{{{\Psi _T}}}} {P_T}, \label{theorem1-proof-eq6}
\end{flalign}
which gives \eqref{theorem1-eq2}.

($\mathbf{ii}$)
Combining \eqref{theorem1-eq1}, \eqref{lemma5-eq2}, \eqref{theorem1-proof-eq1}--\eqref{theorem1-proof-eq5}  yields
\begin{flalign}
\nonumber
& \;\;\;\;\;\mathbf{E}[\frac{1}{n}\sum\limits_{i = 1}^n {{\| {\sum\limits_{t = 1}^T {{{[ {{g_t}( {{x_{i,t}}} )} ]}_ + }} } \|^2}}] \\
\nonumber
&\;\;  \le 4( {p + 1} )n{F_1}{F_2}{\varpi _2}T + 2n{F_1}{F_2}( {{{\hat \varpi }_2} + 2} )T{\Psi _T} + 2( {\frac{1}{{{\gamma _1}}} + \frac{{{T^{1 - \kappa }}}}{{1 - \kappa }} + 80{{\hat \varpi }_3}\sqrt {T{\Psi _T}} } )\Big( {n{F_1}T}\\
\nonumber
&\;\; + 2( {p + 1} )n{F_2}{\varpi _2} + \frac{{2n\hat \varpi _1^2}}{\kappa }{T^\kappa } + 40n{{\hat \varpi }_3}\sqrt {T{\Psi _T}}  + n{F_2}( {{{\hat \varpi }_2} + 2} ){\Psi _T} + 2nR( \mathbb{X} )\sqrt {T{\Psi _T}} \\
&\;\; + 2\sqrt 2 nR{( \mathbb{X} )^2}\sqrt {\frac{T}{{{\Psi _T}}}}  + 2nHR{( \mathbb{X} )^2}\sqrt {\frac{1}{{{\Psi _2}}}} + n{F_2}\big( {R( \mathbb{X} ) + r( \mathbb{X} )} \big)\big( {\frac{{{{\hat \varpi }_1}}}{\kappa }{T^\kappa } + \log ( T )} \big)\Big). \label{theorem1-proof-eq7}
\end{flalign}

From Cauchy--Schwarz inequality, we have
\begin{flalign}
{\big( {\frac{1}{n}\sum\limits_{i = 1}^n {\| {\sum\limits_{t = 1}^T {{{[ {{g_t}( {{x_{i,t}}} )} ]}_ + }} } \|} } \big)^2}
\le \frac{1}{n}\sum\limits_{i = 1}^n {{\| {\sum\limits_{t = 1}^T {{{[ {{g_t}( {{x_{i,t}}} )} ]}_ + }} } \|^2}}. \label{theorem1-proof-eq8}
\end{flalign}

We have
\begin{flalign}
& \sum\limits_{t = 1}^T {\| {{{[ {{g_t}( {{x_{i,t}}} )} ]}_ + }} \|}  \le {\sum\limits_{t = 1}^T {\| {{{[ {{g_t}( {{x_{i,t}}} )} ]}_ + }} \|} _1}  = {\| {\sum\limits_{t = 1}^T {{{[ {{g_t}( {{x_{i,t}}} )} ]}_ + }} } \|_1} \le \sqrt m \| {\sum\limits_{t = 1}^T {{{[ {{g_t}( {{x_{i,t}}} )} ]}_ + }} } \| \label{theorem1-proof-eq9}
\end{flalign}

Combining \eqref{theorem1-proof-eq7}--\eqref{theorem1-proof-eq9} yields \eqref{theorem1-eq3}.
\end{proof}

\hspace{-3mm}\emph{C. Proof of Theorem 2}

Based on Lemma~5, we are now ready to prove Theorem~2.

For any $T \ge 3$, it holds that
\begin{flalign}
\sum\limits_{t = 1}^T {\frac{1}{t}}  \le 1 + \int\limits_1^T {\frac{1}{t}} dt \le 1 + \log ( T ) \le 2\log ( T ). \label{theorem2-proof-eq1}
\end{flalign}

For any constant $b > 1$ and $T \in {\mathbb{N}_ + }$, there exists a constant $M > 0$ such that
\begin{flalign}
\sum\limits_{t = 1}^T {\frac{1}{{{t^b}}}}  \le M. \label{theorem2-proof-eq2}
\end{flalign}

($\mathbf{i}$)
Combining \eqref{theorem2-eq1} with ${\theta _3} \in ( {{\theta _1},1} )$, \eqref{theorem1-proof-eq1}, \eqref{theorem1-proof-eq4} and \eqref{lemma5-eq1} yields
\begin{flalign}
\nonumber
& \;\;\;\;\;\mathbf{E}[{{\rm{Net}\mbox{-}\rm{Reg}}( {\{ {{x_{i,t}}} \},{y_{[ T ]}}} )}] \\
\nonumber
& \le 2( {p + 1} ){F_2}{\varpi _2} + \frac{{2\hat \varpi _1^2{T^{{\theta _2}}}}}{{{\theta _2}}} + \frac{{10{{\hat \varpi }_3}{\alpha _0}{T^{1 - {\theta _1}}}}}{{1 - {\theta _1}}} + \frac{{2R( \mathbb{X} ){\tau _0}{T^{1 + {\theta _1} - {\theta _3}}}}}{{( {1 + {\theta _1} - {\theta _3}} ){\alpha _0}}} + \frac{{{F_2}( {{{\hat \varpi }_2} + 2} ){\tau _0}{T^{1 - {\theta _3}}}}}{{1 - {\theta _3}}}  \\
&\;\; + \frac{{4R{{( \mathbb{X} )}^2}{T^{{\theta _1}}}}}{{{\alpha _0}}} + \frac{{2R( \mathbb{X} ){T^{{\theta _1}}}{P_T}}}{{{\alpha _0}}} + \frac{{2R{{( \mathbb{X} )}^2}( {2 - {\theta _1}} )}}{{{\alpha _0}( {1 - {\theta _1}} )}} + {F_2}\big( {R( \mathbb{X} ) + r( \mathbb{X} )} \big)\big( {\frac{{{{\hat \varpi }_1}{T^{{\theta _2}}}}}{{{\theta _2}}} + \log ( T )} \big). \label{theorem2-proof-eq3}
\end{flalign}
Combining \eqref{theorem2-eq1} with ${\theta _3} = 1$, \eqref{theorem1-proof-eq1}, \eqref{theorem1-proof-eq4}, \eqref{theorem2-proof-eq1} and \eqref{lemma5-eq1} yields
\begin{flalign}
\nonumber
& \;\;\;\;\;\mathbf{E}[{{\rm{Net}\mbox{-}\rm{Reg}}( {\{ {{x_{i,t}}} \},{y_{[ T ]}}} )}] \\
\nonumber
& \le 2( {p + 1} ){F_2}{\varpi _2} + \frac{{2\hat \varpi _1^2{T^{{\theta _2}}}}}{{{\theta _2}}} + \frac{{10{{\hat \varpi }_3}{\alpha _0}{T^{1 - {\theta _1}}}}}{{1 - {\theta _1}}} + \frac{{2R( \mathbb{X} ){\tau _0}{T^{\theta _1}}}}{{ {\theta _1} {\alpha _0}}} + 2{F_2}({\hat \varpi _2} + 2){\tau _0}\log ( T ) \\
&\;\; + \frac{{4R{{( \mathbb{X} )}^2}{T^{{\theta _1}}}}}{{{\alpha _0}}} + \frac{{2R( \mathbb{X} ){T^{{\theta _1}}}{P_T}}}{{{\alpha _0}}} + \frac{{2R{{( \mathbb{X} )}^2}( {2 - {\theta _1}} )}}{{{\alpha _0}( {1 - {\theta _1}} )}}  + {F_2}\big( {R( \mathbb{X} ) + r( \mathbb{X} )} \big)\big( {\frac{{{{\hat \varpi }_1}{T^{{\theta _2}}}}}{{{\theta _2}}} + \log ( T )} \big).  \label{theorem2-proof-eq4}
\end{flalign}
Combining \eqref{theorem2-eq1} with $1 < {\theta _3} < 1 + {\theta _1}$, \eqref{theorem1-proof-eq1}, \eqref{theorem1-proof-eq4}, \eqref{theorem2-proof-eq2} and \eqref{lemma5-eq1} yields
\begin{flalign}
\nonumber
& \;\;\;\;\;\mathbf{E}[{{\rm{Net}\mbox{-}\rm{Reg}}( {\{ {{x_{i,t}}} \},{y_{[ T ]}}} )}] \\
\nonumber
& \le 2( {p + 1} ){F_2}{\varpi _2} + \frac{{2\hat \varpi _1^2{T^{{\theta _2}}}}}{{{\theta _2}}} + \frac{{10{{\hat \varpi }_3}{\alpha _0}{T^{1 - {\theta _1}}}}}{{1 - {\theta _1}}} + \frac{{2R( \mathbb{X} ){\tau _0}{T^{1 + {\theta _1} - {\theta _3}}}}}{{( {1 + {\theta _1} - {\theta _3}} ){\alpha _0}}}  + {F_2}({{\hat \varpi }_2} + 2){\tau _0}M \\
&\;\; + \frac{{4R{{( \mathbb{X} )}^2}{T^{{\theta _1}}}}}{{{\alpha _0}}} + \frac{{2R( \mathbb{X} ){T^{{\theta _1}}}{P_T}}}{{{\alpha _0}}} + \frac{{2R{{( \mathbb{X} )}^2}( {2 - {\theta _1}} )}}{{{\alpha _0}( {1 - {\theta _1}} )}} + {F_2}\big( {R( \mathbb{X} ) + r( \mathbb{X} )} \big)\big( {\frac{{{{\hat \varpi }_1}{T^{{\theta _2}}}}}{{{\theta _2}}} + \log ( T )} \big).  \label{theorem2-proof-eq5}
\end{flalign}
Combining \eqref{theorem2-eq1} with ${\theta _3} = 1 + {\theta _1}$, \eqref{theorem1-proof-eq1}, \eqref{theorem1-proof-eq4}, \eqref{theorem2-proof-eq1}, \eqref{theorem2-proof-eq2} and \eqref{lemma5-eq1} yields
\begin{flalign}
\nonumber
& \;\;\;\;\;\mathbf{E}[{{\rm{Net}\mbox{-}\rm{Reg}}( {\{ {{x_{i,t}}} \},{y_{[ T ]}}} )}] \\
\nonumber
& \le 2( {p + 1} ){F_2}{\varpi _2} + \frac{{2\hat \varpi _1^2{T^{{\theta _2}}}}}{{{\theta _2}}} + \frac{{10{{\hat \varpi }_3}{\alpha _0}{T^{1 - {\theta _1}}}}}{{1 - {\theta _1}}} + \frac{{4R(\mathbb{X}){\tau _0}\log ( T )}}{{{\alpha _0}}} + {F_2}({{\hat \varpi }_2} + 2){\tau _0}M \\
&\;\; + \frac{{4R{{( \mathbb{X} )}^2}{T^{{\theta _1}}}}}{{{\alpha _0}}} + \frac{{2R( \mathbb{X} ){T^{{\theta _1}}}{P_T}}}{{{\alpha _0}}} + \frac{{2R{{( \mathbb{X} )}^2}( {2 - {\theta _1}} )}}{{{\alpha _0}( {1 - {\theta _1}} )}} + {F_2}\big( {R( \mathbb{X} ) + r( \mathbb{X} )} \big)\big( {\frac{{{{\hat \varpi }_1}{T^{{\theta _2}}}}}{{{\theta _2}}} + \log ( T )} \big).  \label{theorem2-proof-eq6}
\end{flalign}
Combining \eqref{theorem2-eq1} with ${\theta _3} > 1 + {\theta _1}$, \eqref{theorem1-proof-eq1}, \eqref{theorem1-proof-eq4}, \eqref{theorem2-proof-eq2} and \eqref{lemma5-eq1} yields
\begin{flalign}
\nonumber
& \;\;\;\;\;\mathbf{E}[{{\rm{Net}\mbox{-}\rm{Reg}}( {\{ {{x_{i,t}}} \},{y_{[ T ]}}} )}] \\
\nonumber
& \le 2( {p + 1} ){F_2}{\varpi _2} + \frac{{2\hat \varpi _1^2{T^{{\theta _2}}}}}{{{\theta _2}}} + \frac{{10{{\hat \varpi }_3}{\alpha _0}{T^{1 - {\theta _1}}}}}{{1 - {\theta _1}}} + \frac{{2R(X){\tau _0}M}}{{{\alpha _0}}} + {F_2}({{\hat \varpi }_2} + 2){\tau _0}M \\
&\;\; + \frac{{4R{{( \mathbb{X} )}^2}{T^{{\theta _1}}}}}{{{\alpha _0}}} + \frac{{2R( \mathbb{X} ){T^{{\theta _1}}}{P_T}}}{{{\alpha _0}}} + \frac{{2R{{( \mathbb{X} )}^2}( {2 - {\theta _1}} )}}{{{\alpha _0}( {1 - {\theta _1}} )}} + {F_2}\big( {R( \mathbb{X} ) + r( \mathbb{X} )} \big)\big( {\frac{{{{\hat \varpi }_1}{T^{{\theta _2}}}}}{{{\theta _2}}} + \log ( T )} \big).  \label{theorem2-proof-eq7}
\end{flalign}

From \eqref{theorem2-proof-eq3}--\eqref{theorem2-proof-eq7}, we have \eqref{theorem2-eq2}.

($\mathbf{ii}$)
Combining \eqref{theorem2-eq1} with ${\theta _3} \in ( {{\theta _1},1} )$, \eqref{theorem1-proof-eq1}, \eqref{theorem1-proof-eq4} and \eqref{lemma5-eq2} yields
\begin{flalign}
\nonumber
& \;\;\;\;\;\mathbf{E}[\frac{1}{n}\sum\limits_{i = 1}^n {{\| {\sum\limits_{t = 1}^T {{{[ {{g_t}( {{x_{i,t}}} )} ]}_ + }} } \|^2}}] \\
\nonumber
&  \le 4( {p + 1} )n{F_1}{F_2}{\varpi _2}T  + \frac{{2n{F_1}{F_2}( {{{\hat \varpi }_2} + 2} ){\tau _0}{T^{2 - {\theta _3}}}}}{{1 - {\theta _3}}} + 2n( {\frac{1}{{{\gamma _1}}} + \frac{{{T^{1 - {\theta _2}}}}}{{1 - {\theta _2}}} + \frac{{40{{\hat \varpi }_3}{\alpha _0}{T^{1 - {\theta _1}}}}}{{1 - {\theta _1}}}} )\Big( {{F_1}T}\\
\nonumber
&  \;\; + 2( {p + 1} ){F_2}{\varpi _2} + \frac{{2\hat \varpi _1^2{T^{{\theta _2}}}}}{{{\theta _2}}} + \frac{{20{{\hat \varpi }_3}{\alpha _0}{T^{1 - {\theta _1}}}}}{{1 - {\theta _1}}} + \frac{{4R{{( \mathbb{X} )}^2}{T^{{\theta _1}}}}}{{{\alpha _0}}} + \frac{{2R{{( \mathbb{X} )}^2}( {2 - {\theta _1}} )}}{{{\alpha _0}( {1 - {\theta _1}} )}} \\
&  \;\; + \frac{{2R( \mathbb{X} ){\tau _0}{T^{1 + {\theta _1} - {\theta _3}}}}}{{{\alpha _0}( {1 + {\theta _1} - {\theta _3}} )}} + \frac{{{F_2}( {{{\hat \varpi }_2} + 2} ){\tau _0}{T^{1 - {\theta _3}}}}}{{1 - {\theta _3}}} + {F_2}\big( {R( \mathbb{X} ) + r( \mathbb{X} )} \big)\big( {\frac{{{{\hat \varpi }_1}{T^{{\theta _2}}}}}{{{\theta _2}}} + \log ( T )} \big)\Big). \label{theorem2-proof-eq8}
\end{flalign}
Combining \eqref{theorem2-eq1} with ${\theta _3} = 1$, \eqref{theorem1-proof-eq1}, \eqref{theorem1-proof-eq4}, \eqref{theorem2-proof-eq1} and \eqref{lemma5-eq2} yields
\begin{flalign}
\nonumber
& \;\;\;\;\;\mathbf{E}[\frac{1}{n}\sum\limits_{i = 1}^n {{\| {\sum\limits_{t = 1}^T {{{[ {{g_t}( {{x_{i,t}}} )} ]}_ + }} } \|^2}}] \\
\nonumber
&  \le 4( {p + 1} )n{F_1}{F_2}{\varpi _2}T + 4n{F_1}{F_2}({{\hat \varpi }_2} + 2){\tau _0}T\log ( T ) + 2n( {\frac{1}{{{\gamma _1}}} + \frac{{{T^{1 - {\theta _2}}}}}{{1 - {\theta _2}}} + \frac{{40{{\hat \varpi }_3}{\alpha _0}{T^{1 - {\theta _1}}}}}{{1 - {\theta _1}}}} )\Big( {{F_1}T}\\
\nonumber
&  \;\;  + 2( {p + 1} ){F_2}{\varpi _2} + \frac{{2\hat \varpi _1^2{T^{{\theta _2}}}}}{{{\theta _2}}} + \frac{{20{{\hat \varpi }_3}{\alpha _0}{T^{1 - {\theta _1}}}}}{{1 - {\theta _1}}} + \frac{{4R{{( \mathbb{X} )}^2}{T^{{\theta _1}}}}}{{{\alpha _0}}} + \frac{{2R{{( \mathbb{X} )}^2}( {2 - {\theta _1}} )}}{{{\alpha _0}( {1 - {\theta _1}} )}} \\
&  \;\; + \frac{{2R(\mathbb{X}){\tau _0}{T^{{\theta _1}}}}}{{{\alpha _0}{\theta _1}}} + 2{F_2}({{\hat \varpi }_2} + 2){\tau _0}\log ( T )  + {F_2}\big( {R( \mathbb{X} ) + r( \mathbb{X} )} \big)\big( {\frac{{{{\hat \varpi }_1}{T^{{\theta _2}}}}}{{{\theta _2}}} + \log ( T )} \big)\Big). \label{theorem2-proof-eq9}
\end{flalign}
Combining \eqref{theorem2-eq1} with $1 < {\theta _3} < 1 + {\theta _1}$, \eqref{theorem1-proof-eq1}, \eqref{theorem1-proof-eq4}, \eqref{theorem2-proof-eq2} and \eqref{lemma5-eq2} yields
\begin{flalign}
\nonumber
& \;\;\;\;\;\mathbf{E}[\frac{1}{n}\sum\limits_{i = 1}^n {{\| {\sum\limits_{t = 1}^T {{{[ {{g_t}( {{x_{i,t}}} )} ]}_ + }} } \|^2}}] \\
\nonumber
&  \le 4( {p + 1} )n{F_1}{F_2}{\varpi _2}T + 4n{F_1}{F_2}({{\hat \varpi }_2} + 2){\tau _0}MT + 2n( {\frac{1}{{{\gamma _1}}} + \frac{{{T^{1 - {\theta _2}}}}}{{1 - {\theta _2}}} + \frac{{40{{\hat \varpi }_3}{\alpha _0}{T^{1 - {\theta _1}}}}}{{1 - {\theta _1}}}} )\Big( {{F_1}T}\\
\nonumber
&  \;\;  + 2( {p + 1} ){F_2}{\varpi _2}  + \frac{{2\hat \varpi _1^2{T^{{\theta _2}}}}}{{{\theta _2}}} + \frac{{20{{\hat \varpi }_3}{\alpha _0}{T^{1 - {\theta _1}}}}}{{1 - {\theta _1}}} + \frac{{4R{{( \mathbb{X} )}^2}{T^{{\theta _1}}}}}{{{\alpha _0}}} + \frac{{2R{{( \mathbb{X} )}^2}( {2 - {\theta _1}} )}}{{{\alpha _0}( {1 - {\theta _1}} )}} \\
&  \;\;  + \frac{{2R( \mathbb{X} ){\tau _0}{T^{1 + {\theta _1} - {\theta _3}}}}}{{{\alpha _0}( {1 + {\theta _1} - {\theta _3}} )}} + {F_2}({{\hat \varpi }_2} + 2){\tau _0}M + {F_2}\big( {R( \mathbb{X} ) + r( \mathbb{X} )} \big)\big( {\frac{{{{\hat \varpi }_1}{T^{{\theta _2}}}}}{{{\theta _2}}} + \log ( T )} \big)\Big). \label{theorem2-proof-eq10}
\end{flalign}
Combining \eqref{theorem2-eq1} with ${\theta _3} = 1 + {\theta _1}$, \eqref{theorem1-proof-eq1}, \eqref{theorem1-proof-eq4}, \eqref{theorem2-proof-eq1}, \eqref{theorem2-proof-eq2} and \eqref{lemma5-eq2} yields
\begin{flalign}
\nonumber
& \;\;\;\;\;\mathbf{E}[\frac{1}{n}\sum\limits_{i = 1}^n {{\| {\sum\limits_{t = 1}^T {{{[ {{g_t}( {{x_{i,t}}} )} ]}_ + }} } \|^2}}] \\
\nonumber
&  \le 4( {p + 1} )n{F_1}{F_2}{\varpi _2}T + 4n{F_1}{F_2}({{\hat \varpi }_2} + 2){\tau _0}MT + 2n( {\frac{1}{{{\gamma _1}}} + \frac{{{T^{1 - {\theta _2}}}}}{{1 - {\theta _2}}} + \frac{{40{{\hat \varpi }_3}{\alpha _0}{T^{1 - {\theta _1}}}}}{{1 - {\theta _1}}}} )\Big( {{F_1}T}\\
\nonumber
&  \;\;  + 2( {p + 1} ){F_2}{\varpi _2} + \frac{{2\hat \varpi _1^2{T^{{\theta _2}}}}}{{{\theta _2}}} + \frac{{20{{\hat \varpi }_3}{\alpha _0}{T^{1 - {\theta _1}}}}}{{1 - {\theta _1}}} + \frac{{4R{{( \mathbb{X} )}^2}{T^{{\theta _1}}}}}{{{\alpha _0}}} + \frac{{2R{{( \mathbb{X} )}^2}( {2 - {\theta _1}} )}}{{{\alpha _0}( {1 - {\theta _1}} )}}\\
&  \;\; + \frac{{4R(\mathbb{X}){\tau _0}\log ( T )}}{{{\alpha _0}}} + {F_2}({{\hat \varpi }_2} + 2){\tau _0}M + {F_2}\big( {R( \mathbb{X} ) + r( \mathbb{X} )} \big)\big( {\frac{{{{\hat \varpi }_1}{T^{{\theta _2}}}}}{{{\theta _2}}} + \log ( T )} \big)\Big). \label{theorem2-proof-eq11}
\end{flalign}
Combining \eqref{theorem2-eq1} with ${\theta _3} > 1 + {\theta _1}$, \eqref{theorem1-proof-eq1}, \eqref{theorem1-proof-eq4}, \eqref{theorem2-proof-eq2} and \eqref{lemma5-eq2} yields
\begin{flalign}
\nonumber
& \;\;\;\;\;\mathbf{E}[\frac{1}{n}\sum\limits_{i = 1}^n {{\| {\sum\limits_{t = 1}^T {{{[ {{g_t}( {{x_{i,t}}} )} ]}_ + }} } \|^2}}] \\
\nonumber
&  \le 4( {p + 1} )n{F_1}{F_2}{\varpi _2}T + 4n{F_1}{F_2}({{\hat \varpi }_2} + 2){\tau _0}MT + 2n( {\frac{1}{{{\gamma _1}}} + \frac{{{T^{1 - {\theta _2}}}}}{{1 - {\theta _2}}} + \frac{{40{{\hat \varpi }_3}{\alpha _0}{T^{1 - {\theta _1}}}}}{{1 - {\theta _1}}}} )\Big( {{F_1}T}\\
\nonumber
&  \;\;  + 2( {p + 1} ){F_2}{\varpi _2} + \frac{{2\hat \varpi _1^2{T^{{\theta _2}}}}}{{{\theta _2}}} + \frac{{20{{\hat \varpi }_3}{\alpha _0}{T^{1 - {\theta _1}}}}}{{1 - {\theta _1}}} + \frac{{4R{{( \mathbb{X} )}^2}{T^{{\theta _1}}}}}{{{\alpha _0}}} + \frac{{2R{{( \mathbb{X} )}^2}( {2 - {\theta _1}} )}}{{{\alpha _0}( {1 - {\theta _1}} )}}\\
&  \;\;  + \frac{{2R(\mathbb{X}){\tau _0}M}}{{{\alpha _0}}} + {F_2}({{\hat \varpi }_2} + 2){\tau _0}M  + {F_2}\big( {R( \mathbb{X} ) + r( \mathbb{X} )} \big)\big( {\frac{{{{\hat \varpi }_1}{T^{{\theta _2}}}}}{{{\theta _2}}} + \log ( T )} \big)\Big). \label{theorem2-proof-eq12}
\end{flalign}

Combining \eqref{theorem1-proof-eq8}, \eqref{theorem1-proof-eq9} and \eqref{theorem2-proof-eq8}--\eqref{theorem2-proof-eq12} yields \eqref{theorem2-eq3}.

\bibliographystyle{IEEEtran}
\bibliography{reference_online}

\end{document}